\documentclass[11pt,twoside,english,reqno,a4paper]{article}

\usepackage{a4wide}
\usepackage[utf8]{inputenc}
\usepackage{amsmath}
\usepackage{amsthm}
\usepackage{amssymb}
\usepackage[numbers]{natbib}
\usepackage{mathrsfs}
\usepackage{lmodern}
\usepackage[T1]{fontenc}
\usepackage[english]{babel}
\usepackage{url}
\usepackage{mathtools,amsmath}
\usepackage{hyperref}
\usepackage{color}
\usepackage{epic}
\usepackage{pict2e}
\usepackage{paralist}
\usepackage{enumerate}
\usepackage{enumitem}
\usepackage{float}
\usepackage{mathtools}
\usepackage{blindtext}
\usepackage{framed}
\usepackage{advdate}
\usepackage{relsize}

\usepackage{titletoc}

\newcommand{\footremember}[2]{%
	\footnote{#2}
	\newcounter{#1}
	\setcounter{#1}{\value{footnote}}%
}

\DeclareMathOperator{\SO}{SO}
\DeclareMathOperator{\GL}{GL}
\DeclareMathOperator{\Orm}{O}
\DeclareMathOperator{\tr}{tr}

\DeclareMathOperator{\Hom}{{Hom}}
\DeclareMathOperator{\Ind}{Ind}
\DeclareMathOperator{\diag}{diag}

\DeclareMathOperator{\End}{End}
\DeclareMathOperator{\pr}{pr}

\newcommand{\nfrak}{\mathfrak{n}}
\newcommand{\nbar}{\bar{\mathfrak{n}}}
\newcommand{\Nbar}{\overline{N}}

\newcommand{\R}{\mathbb{R}}
\newcommand{\C}{\mathbb{C}}
\newcommand{\g}{\mathfrak{g}}

\newcommand{\Z}{\mathbb{Z}}

\renewcommand{\Im}{\operatorname{Im}}
\renewcommand{\Re}{\operatorname{Re}}

\renewcommand{\setminus}{-}
\DeclareMathOperator{\sgn}{sgn}
\DeclareMathOperator{\Res}{Res}

\newcommand{\LocalPoles}{\;\backslash\kern-0.8em{\backslash} \;}
\newcommand{\GlobalPoles}{\; /\kern-0.8em{/} \; }

\DeclarePairedDelimiter\abs{\lvert}{\rvert}
\DeclarePairedDelimiter\norm{\lVert}{\rVert}

\newtheoremstyle{normal}
{10pt}
{10pt}
{\normalfont}
{}
{\bfseries}
{}
{0.8em}
{\bfseries{\thmname{#1}\thmnumber{ #2}.\thmnote{ \hspace{0.5em}(#3)\newline}}}
 
\newtheoremstyle{kursiv}
{10pt}
{10pt}
{\itshape}
{}
{\bfseries}
{}
{0.8em}
{\bfseries{\thmname{#1}\thmnumber{ #2}.\thmnote{ \hspace{0.5em}(#3)\newline}}}

\theoremstyle{plain}
\newtheorem{theorem}{Theorem}[section]
\newtheorem{lemma}[theorem]{Lemma}
\newtheorem{corollary}[theorem]{Corollary}
\newtheorem{prop}[theorem]{Proposition}
\newtheorem{theoremalph}{Theorem}

\theoremstyle{definition}

\newtheorem{remark}[theorem]{Remark}

\numberwithin{equation}{section}

\author{Clemens Weiske\footremember{asd}{Chalmers University of Technology, \texttt{weiske@chalmers.se}}}
\title{Branching of unitary \texorpdfstring{$\Orm(1,n+1)$}{O(1,n+1)}-representations with non-trivial \texorpdfstring{$(\mathfrak{g},K)$}{(g,K)}-cohomology}
\date{ }
\begin{document}
	\maketitle 

		\renewcommand{\thefootnote}{\fnsymbol{footnote}} 
	\renewcommand{\thefootnote}{\arabic{footnote}} 
\abstract{Let $G={\rm O}(1,n+1)$ with maximal compact subgroup $K$ and let	$\Pi$ be a unitary irreducible representation of $G$ with non-trivial $(\mathfrak{g},K)$-cohomology. Then $\Pi$ occurs inside a principal series representation of $G$, induced from the ${\rm O}(n)$-representation $\bigwedge\nolimits^p(\C^n)$ and characters of a minimal parabolic subgroup of $G$ at the limit of the complementary series. Considering the subgroup $G'={\rm O}(1,n)$ of $G$ with maximal compact subgroup $K'$, we prove branching laws and explicit Plancherel formulas for the restrictions to $G'$ of all unitary representations occurring in such principal series, including the complementary series, all unitary $G$-representations with non-trivial $(\mathfrak{g},K)$-cohomology and further relative discrete series representations in the cases $p=0,n$. Discrete spectra are constructed explicitly as residues of $G'$-intertwining operators which resemble the Fourier transforms on vector bundles over the Riemannian symmetric space $G'/K'$.}

\section*{Introduction}
Unitary representations of reductive Lie groups with non-trivial $(\mathfrak{g},K)$-cohomology appear in several branches of mathematics, for example in the theory of locally symmetric spaces, where for a Lie-group $G$ with finite center, maximal compact subgroup $K$ and discrete cocompact subgroup $\Gamma$, by the Matsushima--Murakami formula (see \cite[VII, Theorem~3.2]{borel_wallach_1980})
$$H^*(\Gamma\backslash G / K,\C)=\bigoplus_{\pi \in \widehat{G}}m(\Gamma,\pi)H^*(\mathfrak{g},K;\pi_K)$$
the cohomology of $\Gamma\backslash G / K$ is given by $(\mathfrak{g},K)$-cohomologies of unitary representations of $G$ with multiplicities, which are essentially the dimensions of spaces of automorphic forms on $\Gamma\backslash G /K$. All representations with non-trivial $(\mathfrak{g},K)$-cohomology are constructed and their cohomologies calculated in \cite{vogan_zuck_1984}. The unitary ones are well known for the indefinite orthogonal group $\Orm(1,n+1)$ and classified for example for $\GL(n,\R)$ (see \cite{speh_1983}). In our case of interest $\Orm(1,n+1)$ the unitary cohomological representations occur as limits of complementary series representations in the Fell topology on the unitary dual and on the automorphic dual in the sense of Burger--Sarnak \cite{burger_sarnak_1991}. Restrictions of these representations are of particular importance in the latter setting, since by \cite{burger_sarnak_1991}, the restriction of automorphic representations to certain subgroups is again automorphic for the subgroup. In \cite{speh_venkataramana_2011} it is proven that the restrictions of the cohomological representations of $\Orm(1,n+1)$ to $\Orm(1,n)$ contain a certain cohomological representation of the subgroup discretely. The main result of this article is the full branching law for the cohomological representations of $\Orm(1,n+1)$ restricted to $\Orm(1,n)$, extending the result of \cite{speh_venkataramana_2011} to a complete decomposition.

For an irreducible unitary representation $\pi$ of a reductive Lie group $G$ which is typically infinite dimensional, the restriction to a subgroup $G'$ decomposes into a direct integral
$$\pi|_{G'} \simeq \int_{\widehat{G}'}^\oplus m(\pi,\tau)\tau\, d\mu_\pi(\tau)$$
with a certain measure $d\mu_\pi$ on the unitary dual $\widehat{G}'$ of $G'$ and possibly infinite multiplicities $m(\pi,\tau)$. Only in special cases the support of the measure is discrete and in general it might contain a continuous and a discrete part. Many special cases have been studied recently using analytic methods (e.g. \cite{Kob19}, \cite{OS19}, \cite{SZ16}, \cite{MO15}). 

For pointwise evaluation of the continuous spectrum of a unitary branching law in terms of $G'$-intertwining operators, it is necessary to restrict ourselves to the smooth vectors of  unitary representations, since the existence of  continuous $G'$-intertwining operators in the unitary case already implies the images to be in the discrete spectrum, while for the smooth vectors, intertwiners exist for the whole spectrum almost everywhere (see \cite{frahm_2020}).
In our case the cohomological representations of $\Orm(1,n+1)$ can be realized as quotients of principal series representations induced from the $\Orm(n)$-representation $\bigwedge\nolimits^p(\C^n)$ at the limit of the complementary series. To obtain the direct integral decomposition of the cohomological representations, we prove branching laws for the unitary principal series and use an analytic continuation procedure to extend the result onto the complementary series and towards the cohomological representations. In particular we obtain branching laws for the complementary series and also for all other unitarizable quotients which occur within the principal series in question. More precisely we collect discrete components in the decomposition as residues of $G'$-intertwining operators between smooth vectors, so called \emph{symmetry breaking operators} by Kobayashi \cite{kobayashi_2015} and we make use of the detailed classification and study of these operators in the relevant case by Kobayashi--Speh \cite{kobayashi_speh_2018}.

\subsection*{Main results}
Let $G=\Orm(1,n+1)$, $n>1$ and let $P=MAN\subseteq G$ be a minimal parabolic subgroup. Then $M\cong\Orm(1)\times \Orm(n)$. Consider the representation
$$\left(\alpha\otimes \bigwedge\nolimits^p(\C^n)\right) \otimes e^\lambda \otimes \mathbf{1}$$
of $MAN$ on the vector space $V_{p,\lambda}^\pm$ where we use the superscript $+$ if $\alpha$ is the trivial irreducible $\Orm(1)$-representation and $-$ if it is the non-trivial one and $\lambda \in \mathfrak{a}_\C^*$ which we identify by $\C$ by mapping the half sum of all positive roots $\rho$ to $\frac{n}{2}$.
Let $\pi_{p,\lambda}^\pm$ be the principal series representation of $G$ on the smooth sections of the homogeneous bundle 
$$G\times_P V_{p,\lambda+\rho}^\pm\to G/P$$
over the real flag variety $G/P$. Our normalization is chosen such that $\pi_{p,\lambda}^\pm$ is unitary for $\lambda \in i\R$ and such that $\pi_{p,\lambda}^\pm$ contains a submodule $\Pi_{p,\pm}$ whose underlying $(\mathfrak{g},K)$-module has non-trivial $(\mathfrak{g},K)$-cohomology for $\lambda =p-\rho$.

Let $G'=\Orm(1,n)$ embedded in $G$ such that $P'=G'\cap P$ is a minimal parabolic subgroup of $G'$. 
Similarly we consider the $P'=M'AN'$ representation
$$\left(\alpha\otimes \bigwedge\nolimits^q(\C^{n-1})\right) \otimes e^\nu \otimes \mathbf{1}$$
on the vector space $W_{q,\nu}^\pm$ and denote by $\tau_{q,\nu}^\pm$ the principal series representation which is given by the smooth sections of the bundle
$$G'\times_{P'} W_{q,\nu+\rho'}^\pm\to G'/P',$$
where $\rho'$ is the obvious and under the identification above equal to $\frac{n-1}{2}$. Our normalization is again such that the unitary principal series is given on the imaginary axis and such that $\tau_{q,\nu}^\pm$ contains a cohomological representation $\Pi'_{q,\pm}$ as a submodule for $\nu=q-\rho'$.

For  $G$ and $G'$ we denote the unitary closures of unitarizable representations $\pi$ in the following by $\hat{\pi}$.
For the unitary principal series we prove the following branching laws. For the uniform formulation for all $p=0,\dots,n$ we set $\hat{\tau}_{q,\nu}^\pm=\{0\}$ for $q=-1,n$.
\begin{theoremalph}[Branching laws for the unitary principal series (see Lemma~\ref{C:lemma:branching_law_unitary_ps})]
	For $\lambda \in i\R $  and $p\neq \frac{n}{2}$ we have
	$$\hat{\pi}_{p,\lambda}^\pm|_{G'}\simeq \bigoplus_{\alpha=+,-}\bigoplus_{q=p-1,p}\int^\oplus_{i\R_+} \hat{\tau}_{q,\nu}^\alpha \, d\nu.$$

	For $\lambda \in i\R  $  and $p=\frac{n}{2}$ we have
	$$\hat{\pi}_{p,\lambda}^\pm|_{G'}\simeq \widehat{\Pi}'_{\frac{n}{2},+} \oplus \widehat{\Pi}'_{\frac{n}{2},-} \oplus \bigoplus_{\alpha=+,-}\bigoplus_{q=p-1,p}\int^\oplus_{i\R_+} \hat{\tau}_{q,\nu}^\alpha \, d\nu .$$
\end{theoremalph}
If $p \neq \rho=\frac{n}{2}$, there is a complementary series. More precisely $\pi_{p,\lambda}^\pm$ is a complementary series representation if and only if $\lambda \in (-\abs{\rho-p}, \abs{\rho-p})$, for which we also prove unitary branching laws and where complementary series of $G'$ occur discretely. We formulate the result only for the negative half of the complementary series. The result for the positive parameters follows by duality.

\begin{theoremalph}[see Theorem~\ref{C:theorem:branching_complementary series}]
For $\lambda\in (-\abs{\rho-p},0)$ we have
$$\hat{\pi}_{p,\lambda}^\pm|_{G'}\simeq \bigoplus_{\alpha=+,-}\bigoplus_{q=p-1,p} \left(\int^\oplus_{i\R_+} \hat{\tau}_{q,\nu}^\alpha\, d\nu \oplus \bigoplus_{k \in [0, \frac{-\lambda-1+(\alpha \frac{1}{2})}{2})\cap \Z} \hat{\tau}_{q,\lambda+1-(\alpha\frac{1}{2})+2k}^{\pm \alpha}\right).$$
\end{theoremalph}
Moreover we prove unitary branching laws for the unitarizable quotients $\Pi_{p,\pm}$ whose underlying $(\mathfrak{g},K)$-modules have non-trivial $(\mathfrak{g},K)$-cohomology, sitting as quotients at the limit of the complementary series. Here complementary series as well as cohomological representations occur in the discrete spectrum.

\begin{theoremalph}[see Theorem~\ref{C:theorem:branching_infinitesimal_char_rho}]
	\begin{enumerate}[label=(\roman{*})]
		\item For the one dimensional quotients we have
		$$\widehat{\Pi}_{0,\pm}|_{G'} \simeq  \widehat{\Pi}'_{0,\pm}, \qquad \widehat{\Pi}_{n+1,\pm}|_{G'}\simeq \widehat{\Pi}'_{n,\pm}.$$
		\item For $0<p\leq \frac{n}{2}$ we have
		$$\widehat{\Pi}_{p,\pm}|_{G'} \simeq \widehat{\Pi}'_{p,\pm}\oplus \bigoplus_{k\in (0,\rho'-p+1)\cap \Z} \hat{\tau}_{p-1,p-1-\rho'+k}^{\mp(-1)^k}\oplus \bigoplus_{\alpha=+,-} \int_{i\R_+}^\oplus \hat{\tau}_{p-1,\nu}^\alpha\, d\nu .
		$$

		\item For $n$ odd and $p= \frac{n+1}{2}$ we have
		$$\widehat{\Pi}_{\frac{n+1}{2},\pm}|_{G'} \simeq \bigoplus_{\alpha=+,-} \int_{i\R_+}^\oplus \hat{\tau}_{\frac{n-1}{2},\nu}^\alpha\, d\nu .
		$$
		\item For $\frac{n+1}{2}< p\leq n $ we have
		$$\widehat{\Pi}_{p,\pm}|_{G'} \simeq \widehat{\Pi}'_{p-1,\pm}\oplus \bigoplus_{k\in (0,p-1-\rho')\cap \Z} \hat{\tau}_{p-1,\rho'-p+1+k}^{\pm(-1)^k}\oplus \bigoplus_{\alpha=+,-} \int_{i\R_+}^\oplus \hat{\tau}_{p-1,\nu}^\alpha\, d\nu .
		$$
	\end{enumerate}
\end{theoremalph}

Speh--Venkataramana proved the inclusions 
\begin{align*}
	&\widehat{\Pi}'_{p,\pm} \subseteq \widehat{\Pi}_{p,\pm}|_{G'} &&p<\frac{n+1}{2}, \\
	&\widehat{\Pi}'_{p-1,\pm} \subseteq \widehat{\Pi}_{p,\pm}|_{G'}  &&p>\frac{n+1}{2},
\end{align*}
(see \cite[Theorem~1.4]{speh_venkataramana_2011}) and the theorem above gives the full decomposition of the cohomological representations $ \widehat{\Pi}_{p,\pm}|_{G'}$.

For $p\neq0,n$, the representations $\Pi_{p,\pm}$ are the only proper unitarizable composition factors of $\pi_{p,\lambda}^\pm$. For $p=0,n$ there are additional unitarizable composition factors $I_{p,j,\pm}$ for each positive integer $j$, occuring as quotients in $\pi_{p,\lambda}^\pm$ for $\lambda =-\rho-j$. We prove branching laws for the closures of these representations as well. Here complementary series, a cohomological representation, as well as the corresponding quotients for the subgroup $I'_{q,k,\pm}$, with $q=0,n-1$ and $k$ positive integers occur discretely.

\begin{theoremalph}[see Theorem \ref{C:theorem:branchin_p=0,n}]
	\begin{enumerate}[label=(\roman{*})]
		\item For $p=0$ we have
		$$\hat{I}_{0,j,\pm}|_{G'} \simeq \widehat{\Pi}'_{1,\mp} \oplus \bigoplus_{k=1}^j\hat{I}'_{0,k, \pm (-1)^{k+j} }   \oplus \bigoplus_{k \in (0,\rho') \cap \Z} \hat{\tau}_{0,-\rho'+k}^{\pm(-1)^{k+j}} \oplus \bigoplus_{\alpha=+,-} \int^\oplus_{i\R_+} \hat{\tau}_{0,\nu}^\alpha\, d\nu .$$
		\item For $p=n$ we have
		$$\hat{I}_{n,j,\pm}|_{G'} \simeq \widehat{\Pi}'_{n-1,\pm} \oplus\bigoplus_{k=1}^j \hat{I}'_{n-1,k, \pm (-1)^{k+j} }   \oplus \bigoplus_{k \in (0,\rho') \cap \Z} \hat{\tau}_{n-1,-\rho'+k}^{\pm(-1)^{k+j}} \oplus \bigoplus_{\alpha=+,-} \int^\oplus_{i\R_+} \hat{\tau}_{n-1,\nu}^\alpha\, d\nu .$$
	\end{enumerate}
\end{theoremalph}
For all branching laws we obtain an explicit Plancherel theorem (see Corollary~\ref{C:cor:plancherel}).
We remark that the decompositions in the spherical case, i.e. $p=0$ have been proven before by Möllers--Oshima in \cite{MO15} by a different approach which is likely to generalize to arbitrary $p$. Our method of proof offers a more systematical perspective which does not rely on the nilradical to be abelian.
\subsection*{Method of proof}
The subgroup $G'$ acts on the real flag variety $G/P$ with an open orbit which is as a $G'$-space given by a $\Z/2\Z$-fibration over the Riemannian symmetric space $G'/K'$, where $K'$ is the maximal compact subgroup of $G'$ (see Proposition~\ref{C:cor:orbits} and Lemma~\ref{C:lemma:H-Stab}). Restriction to the open orbit naturally induces a $G'$-map
$$\Phi:\pi_{p,\lambda}^\pm\to L^2 \left(G'/K',\bigwedge\nolimits^p(\C^n)\right)$$
if $\Re \lambda >-\frac{1}{2}$ (see Lemma~\ref{C:lemma:L^2_condition}). The Plancherel and inversion formula for the space $L^2 \left(G'/K',\bigwedge\nolimits^p(\C^n)\right)$ is essentially due to \cite{camporesi_1997}. It is given in terms of Fourier transforms on $L^2\left(G'/K',\bigwedge\nolimits^p(\C^n)\right)$, which are $G'$-intertwining maps into principal series $\tau_{q,\nu}^\pm$. By composition of the map $\Phi$ and the Fourier transforms we obtain elements of the space $\Hom_{G'}(\pi_{p,\lambda}^\pm|_{G'},\tau_{q,\nu}^\pm)$ of symmetry breaking operators, which are in this special case classified by Kobayashi--Speh in \cite{kobayashi_speh_2018}. The symmetry breaking operators we obtain in this procedure are given by families of integral kernel operators with meromorphic dependence on $\lambda$ and $\nu$ and the meromorphic structure of the operators is studied in \cite{kobayashi_speh_2018} in great detail. This allows us to carefully analytically continue the Plancherel formula of $L^2 \left(G'/K',\bigwedge\nolimits^p(\C^n)\right)$ in $\lambda$ over the critical point $\lambda=-\frac{1}{2}$ on the real axis, towards the complementary series and unitarizable quotients $\Pi_{p,\pm}$.

\subsection*{Structure of this article}
In Section~\ref{C:sec:symm_breaking_general} we recall some facts about symmetry breaking operators between principal series representations and establish the necessary notation for principal series representations of $G$ in Section~\ref{C:sec:principal_series_G}. In Section~\ref{C:sec:component_group} we discuss the restriction to the identity component of the representations in question which will be used for arguments later in the article. In Section~\ref{C:sec:composition_series} and Section~\ref{C:sec:unitary_composition} we study the composition series of the principal series representations and give criteria for reducibility and unitarizability.
In Section~\ref{C:sec:classification_sbos} we recall the classification of symmetry breaking operators between $\pi_{p,\lambda}^\pm|_{G'}$ and $\tau_{q,\nu}^\pm$ from \cite{kobayashi_speh_2018} as well as their meromorphic structure. We recall functional equations of symmetry breaking operators and the standard Knapp--Stein intertwining operators in Section~\ref{C:sec:func_equations} and extend them to operators into quotients of principal series representations in Section~\ref{C:sec:sbo_quotients}. We establish the structure of the open $G'$-orbit in $G/P$ as a homogeneous $G'$-space in Section~\ref{C:sec:homogeneous_G'_space} and prove a Plancherel formula for the corresponding space in Section~\ref{C:sec:plancherel_G'} using the Plancherel formula for the restriction to the connected component of \cite{camporesi_1997} (Section~\ref{C:sec:plancherel_G'_0}). We lift these results to the representation $\pi_{p,\lambda}^\pm$ around the unitary axis in Section~\ref{C:sec:branching_laws}. The main result here is Theorem~\ref{C:theorem:coordinate_change}, by which the Fourier transform on the homogeneous $G'$-space is essentially given by symmetry breaking operators classified by Kobayashi--Speh on the principal series. Finally Section~\ref{C:sec:analytic_continuation} is dedicated to the proof of the main theorems where we analytically continue the Plancherel formula around the unitary axis towards the complementary series and the unitary composition factors.

\subsection*{Acknowledgements}
The results of this article have been part of my PhD thesis at Aarhus University and we thank my supervisor Jan Frahm for his help and input.

We thank Toshiyuki Kobayashi for inspiring discussions on the contents of this article.

The author was supported by the DFG project 325558309 and the KAW grant 2020.0275.

\subsection*{Notation}
For two sets $B \subseteq A$ we use the Notation $A \setminus B= \{ a\in A: a \notin B  \}$. We denote Lie groups by Roman capitals and their corresponding Lie algebras by the corresponding Fraktur lower cases. 

\newpage

\section{Symmetry breaking operators between principal series representations}\label{C:sec:symm_breaking_general}

We recall the basic facts about symmetry breaking operators between principal series representations from \cite{kobayashi_speh_2018}.

\subsection{Principal series representations}\label{C:sec:PrincipalSeries}

Let $G$ be a real reductive Lie group and $P$ a minimal parabolic subgroup of $G$ with Langlands decomposition $P=MAN$. For a finite-dimensional representation $(\xi,V)$ of $M$, a character $\lambda \in \mathfrak{a}_\C^*$ and the trivial representation $\mathbf{1}$ of $N$ we obtain a finite-dimensional representation $(\xi \otimes e^\lambda \otimes \mathbf{1}, V_{\xi,\lambda})$ of $P=MAN$. By smooth normalized parabolic induction this representation gives rise to the principal series representation
$$ \pi_{\xi,\lambda}:= \Ind_P^G(\xi \otimes e^\lambda \otimes \mathbf{1}) $$
as the left-regular representation of $G$ on the space
$$\{ \varphi \in C^\infty(G,V): \ \varphi(gman)=\xi(m)^{-1}a^{-(\lambda+\rho)}\varphi(g) \ \forall man\in MAN    \},$$
where $\rho:= \frac{1}{2} \tr \operatorname{ad}|_{\mathfrak{n}} \in \mathfrak{a}^*_\C$.
Let $\mathcal{V}_{\xi,\lambda}:= G \times_P V_{\xi,\lambda+\rho} \to G/P$ be the homogeneous vector bundle associated to $V_{\xi,\lambda+\rho}$. Then $\pi_{\xi,\lambda}$ identifies with the left-regular action of $G$ on the space of smooth sections $C^\infty(G/P,\mathcal{V}_{\xi,\lambda})$.

Now let $G'<G$ be a reductive subgroup. Similarly we let $P'=M'A'N'$ be a minimal parabolic subgroup of $G'$. For a finite-dimensional representation $(\eta,W)$ of $M'$ and $\nu \in (\mathfrak{a}'_\C)^*$ we obtain a finite-dimensional representation $(\eta \otimes e^\nu \otimes \mathbf{1},W_{\eta,\lambda})$ of $P'$ and the corresponding principal series representation
$$ \tau_{\eta,\nu}:= \Ind_{P'}^{G'}(\eta \otimes e^\nu \otimes \mathbf{1}). $$
Again we identify $\tau_{\eta,\nu}$ with the smooth sections $C^\infty(G'/P',\mathcal{W}_{\eta,\nu})$ of the homogeneous vector bundle $\mathcal{W}_{\eta,\nu}:= G' \times_{P'} W_{\eta,\nu+\rho'}\to G'/P'$, where $\rho':= \frac{1}{2}\tr \operatorname{ad}|_{\nfrak'}$.

\subsection{Symmetry breaking operators}

In these realizations the space of symmetry breaking operators between $\pi_{\xi,\lambda}$ and $\tau_{\eta,\nu}$ is given by the continuous linear $G'$-maps between the smooth sections of the two homogeneous vector bundles
$$\Hom_{G'}(\pi_{\xi,\lambda}|_{G'},\tau_{\eta,\nu})=\Hom_{G'}(C^\infty(G/P,\mathcal{V}_{\xi,\lambda})),C^\infty(G'/P',\mathcal{W}_{\eta,\nu})).$$
The Schwartz Kernel Theorem implies that every such operator is given by a $G'$-invariant distribution section of the tensor bundle $\mathcal{V}_{\xi^*,-\lambda}\boxtimes\mathcal{W}_{\eta,\nu}$ over $G/P\times G'/P'$, where $\xi^*$ is the representation contradigent to $\xi$. Since $G'$ acts transitively on $G'/P'$ we can consider these distributions as sections on $G/P$ with a certain $P'$-invariance:

\begin{theorem}[{\cite[Proposition 3.2]{kobayashi_speh_2015}}]
\label{C:theorem:KS_kernel}
There is a natural bijection
$$\Hom_{G'}(\pi_{\xi,\lambda}|_{G'},\tau_{\eta,\nu}) \stackrel{\sim}{\longrightarrow} (\mathcal{D}'(G/P,\mathcal{V}_{\xi^*,-\lambda})\otimes W_{\eta,\nu+\rho'})^{P'}, \quad T\mapsto u^T
.$$
\end{theorem}

In our case of interest the dimension of $\Hom_{G'}(\pi_{\xi,\lambda}|_{G'},\tau_{\eta,\nu})$ is in particular generically bounded by $1$.
\begin{theorem}[\cite{sun_zhu_2012} Theorem B]
	\label{C:theorem:mult_one}
	For $(G,G')=(\Orm(1,n+1),\Orm(1,n))$ we have
	$$\dim \Hom_{G'}(\pi|_{G'},\tau)\leq 1$$
	for all irreducible Casselman--Wallach representations $\pi$ of $G$ and $\tau$ of $G'$.
\end{theorem}

\subsection{Restriction to the open Bruhat cell}\label{C:sec:DistributionKernelsOnOpenBruhatCell}

From now on assume $M'=M\cap G'$, $A'=A\cap G'$ and $N'=N\cap G'$. Let $\Nbar$ be the nilradical of the parabolic opposite to $P$. Since $\Nbar$ is unipotent we obtain a parameterization of the open Bruhat cell $\Nbar P/P \subseteq G/P$ in terms of the Lie algebra $\nbar$ by the map
$$\nbar\stackrel{\exp}{\longrightarrow}  \Nbar \lhook \joinrel \longrightarrow G \longrightarrow G/P
,$$ 
such that we can consider $\nbar$ as an open dense subset of $G/P$. Then the restriction
$$\mathcal{D}'(G/P,\mathcal{V}_{\xi^*,-\lambda})\longrightarrow \mathcal{D}'(\nbar,\mathcal{V}_{\xi^*,-\lambda}|_{\nbar})$$ can be used to define a $\g$-action on $\mathcal{D}'(\nbar,\mathcal{V}_{\xi^*,-\lambda}|_{\nbar})\cong\mathcal{D}'(\nbar)\otimes V_{\xi^*,-\lambda+\rho}$ by vector fields. Moreover, since ${\rm Ad}(M'A')$ leaves $\nbar$ invariant, the restriction is further $M'A'$-equivariant.
If we assume $P'\Nbar P=G$, i.e. every $P'$-orbit in $G/P$ meets the open Bruhat cell $\Nbar P$, then symmetry breaking operators can be described in terms of $(M'A',\mathfrak{n}')$-invariant distributions on $\nbar$:

\begin{theorem}[{\cite[Theorem 3.16]{kobayashi_speh_2015}}]\label{thm:KS-distr-kernel}
\label{C:theorem:KS_kernel_open_cell}
	Assume $P'\Nbar P=G$, then there is a natural bijection
	$$\Hom_{G'}(\pi_{\xi,\lambda}|_{G'},\tau_{\eta,\nu}) \stackrel{\sim}{\longrightarrow} (\mathcal{D}'(\nbar)\otimes V_{\xi^*,-\lambda+\rho}\otimes W_{\eta, \nu+\rho'})^{M'A',\mathfrak{n}'}.$$
\end{theorem}

Given a distribution kernel $u^T$, the corresponding operator $T \in \Hom_{G'}(\pi_{\xi,\lambda}|_{G'},\tau_{\eta,\nu})$ is given by
\begin{equation}\label{eq:kernel_operator}
T \varphi(h)=\langle u^T, \varphi(h\exp(\;\cdot\;)) \rangle.
\end{equation}

\section{Principal series representations of rank one orthogonal groups}\label{C:sec:principal_series_G}

Let $G= \Orm(1,n+1)$ denote the group of $(n+2)\times(n+2)$ matrices over $\R$ preserving the quadratic form
$$(z_0,z_1,\ldots,z_{n+1})\mapsto-|z_0|^2+|z_1|^2+\cdots+|z_{n+1}|^2. $$

Let $P$ be the minimal parabolic subgroup of $G$ with Langlands decomposition $P=MAN$ given by
\begin{align*}
 M &= \left\{ \begin{pmatrix}a & & \\& a & \\& & b\end{pmatrix}: a \in \Orm(1), b \in \Orm(n)\right\},\\
 A &= \exp (\mathfrak{a}) \qquad \text{where } \mathfrak{a}=\R H, \quad H=\begin{pmatrix}0 & 1 & \\1 & 0 & \\& & \mathbf{0}_n\end{pmatrix},\\
 N &= \exp (\mathfrak{n}) \qquad \text{where } \mathfrak{n}=\left\{ \begin{pmatrix}0 & 0 & X \\0 & 0 & X \\X^T & -X^T & \mathbf{0}_n\end{pmatrix}: X\in \R^n \right\}.
\end{align*}
Note that $X\in \R^n$ is considered as a row vector.
We identify $\mathfrak{a}_\C^*\cong\C$ by $\lambda\mapsto\lambda(H)$. Then in particular
$$ \rho= \frac{1}{2} \tr \operatorname{ad}|_\mathfrak{n}(H)= \frac{n}{2}. $$

Consider the finite dimensional representations $$\xi=\alpha\otimes {\sigma_p},$$ of $M=\Orm(1)\times \Orm(n)$, with $\alpha \in \{ \mathbf{1}, \operatorname{sgn}\}\cong
\widehat{\Orm}(1)$ and ${\sigma_p}= \bigwedge\nolimits^p {(\C^n)}$ with $p\in \{ 0, \dots,n \}$.
We define the principal series representations
$$\pi_{p,\lambda}^\pm:=\Ind(\xi \otimes e^\lambda \otimes \mathbf{1})$$
where we use the index $+$ if $\alpha=\mathbf{1}$ and $-$ if $\alpha=\sgn$.

Similarly we consider the finite dimensional $M'$ representations
$$\eta=\alpha\otimes {\delta_q}$$
with $\alpha$ as above and ${\delta_q}=\bigwedge\nolimits^q {(\C^{n-1})}$ with $q\in\{0, \dots, n-1\}$
and denote the corresponding prinicpal series representations by $\tau_{q,\nu}^\pm$.

\subsection{The non-compact picture}

Let $\Nbar$ be the nilradical of the parabolic subgroup opposite to $P$. Since $\Nbar$ is unipotent, we identify it with its Lie algebra $\nbar\cong \R^n$ in terms of the exponential map:
$$
\R^n\to\Nbar, \quad X\mapsto\overline{n}_{X}:=\exp \begin{pmatrix}
0& 0 & X \\
0 & 0& -X \\
X^T & X^T & \mathbf{0}_n
\end{pmatrix}.
$$
Here we consider $X\in \R^n$ again as a row vector.
Since $\Nbar P$ is open and dense in $G$, the restriction of $\pi_{p,\lambda}^\pm$ to functions on $\Nbar$ is one-to-one. The resulting realization in $C^\infty(\Nbar)$ of $\pi_{p,\lambda}^\pm$ is called the \emph{non-compact picture of $\pi_{p,\lambda}^\pm$}. For $g \in \Nbar MAN$ we write $g=\overline{n}(g) m(g) a(g) n(g)$ for the obvious decomposition. Then the $G$-action in the non-compact picture is given by
\begin{equation}
\label{C:eq:action_non_compact_picture}
\pi_{p,\lambda}^\pm(g)f(X)=\xi^{-1}(m(g^{-1}\overline{n}_X))a(g^{-1} \overline{n}_X)^{-(\lambda+\rho)}f(\log\overline{n}(g^{-1}\overline{n}_X)),
\end{equation}
whenever $g^{-1}\overline{n}_X\in\overline{N}MAN$.

Let $\tilde{w}_0=\diag(-1,1,\mathbf{1}_n)$, then $\tilde{w}_0$ represents the longest Weyl group element of $G$ with respect to $A$. The following Lemma is easily verified by standard calculations.

\begin{lemma}
\label{C:prop:matrix_decompositions}
\begin{enumerate}[label=(\roman{*}), ref=\thetheorem(\roman{*})]
\item\label{C:prop:matrix_decompositions:iii} Let $m=\diag(a,a,b^{-1}) \in M$ with $a \in \Orm(1)$ and $b\in \Orm(n)$, then
$$ m \overline{n}_Xm^{-1}=\overline{n}_{aXb}. $$
\item\label{C:prop:matrix_decompositions:ii} Let $t\in \R$ and $a=\exp(tH)$, then
$$ a\overline{n}_Xa^{-1}=\overline{n}_{e^{-t}X}. $$
\item\label{C:prop:matrix_decompositions:i} Let $X \neq 0$, then  $\tilde{w}_0 \overline{n}_{X}=\overline{n}_{U}m a n$ with $n\in N$ and
\begin{equation*}
 U=\frac{-X}{\abs{X}^2}, \qquad a = \exp(2\log(\abs{X})H).
\end{equation*}
$$m=\diag\left(-1,-1,\psi_n(X)\right),$$
with $$\psi_n(X)=\mathbf{1}_n-\frac{2X^TX}{\abs{X}^2} \in \Orm(n)$$.
\end{enumerate}
\end{lemma}

These decompositions immediately imply the following formulas for the action of $P$ and $\tilde{w}_0$:

\begin{prop}
\begin{enumerate}[label=(\roman{*}), ref=\thetheorem(\roman{*})]
\item\label{C:prop:GroupActionM} For $m=\diag(a^{-1},a^{-1},b)\in M$ with $a\in\Orm(1)$ and $b\in\Orm(n)$:
$$ \pi_{p,\lambda}^\pm(m)u(X) =\xi^{-1}(m) u(aXb). $$
\item For $t\in\R$ and $a=\exp(tH)$:
$$ \pi_{p,\lambda}^\pm(a)u(X) = e^{(\lambda+\rho)t}u(e^tX). $$
\item\label{C:prop:GroupActionNbar} For $Y\in\R^n$:
$$ \pi_{p,\lambda}^\pm(\overline{n}_{Y})u(X) = u(X-Y). $$
\item
\label{C:prop:longest_weyl_group_element_action}
For the action of $\tilde{w}_0$ we have
\begin{equation*}
	\pi_{p,\lambda}^\pm(\tilde{w}_0)u(X)=\xi^{-1}(\diag(-1,-1,\psi_n(X)))\abs{X}^{-2(\lambda+\rho)}u({\sigma}(X)).
\end{equation*}
where ${\sigma}: \nbar \setminus \{0 \} \to \nbar \setminus \{0 \}$ is the inversion given by
$$\sigma(X)=\frac{-X}{\abs{X}^2}.
$$
\end{enumerate}
\end{prop}

Note that $X\in\R^n$ is a row vector, so that matrix multiplication is from the right.

\subsection{Orbit structure of \texorpdfstring{$G/P$}{G/P}}

By \cite[Proposition~2.9]{frahm_weiske_2020}, the $P'$-orbits in $G/P$ are given by the following.

\begin{prop}
\label{C:cor:orbits}
The $P'$-orbits in $G/P$ and their closure relations are
$$ \mathcal{O}_A \stackrel{1}{\rule[.5ex]{3em}{.4pt}} \mathcal{O}_B \stackrel{n-1}{\rule[.5ex]{3em}{.4pt}} \mathcal{O}_C,$$
where
\begin{align*}
 \mathcal{O}_A &= P'\cdot\tilde{w}_0\overline{n} P = \tilde{w}_0(\Nbar\setminus\Nbar')P,\\
 \mathcal{O}_B &= P'\cdot\tilde{w}_0P = \tilde{w}_0\Nbar'P,\\
 \mathcal{O}_C &= P'\cdot\mathbf{1}_{n+2}P,
\end{align*}
for some $\overline{n}\in \Nbar \setminus \Nbar'$. Here $X\stackrel{k}{\rule[.5ex]{2em}{.4pt}}Y$ means that $Y$ is a subvariety of $\bar{X}$ of co-dimension $k$.
\end{prop}

In particular the orbit $\mathcal{O}_A$ is open in $G/P$.

\section{The component group \texorpdfstring{$G/G_0$}{G/G0}}\label{C:sec:component_group}
We study the restriction of representations of $G$ to the identity component $G_0$.
\subsection{Global characters of \texorpdfstring{$\Orm(1,n+1)$}{O(1,n+1)}}
$G=\Orm(1,n+1)$ is a disconnected group with four connected components and the identity component $G_0$ is isomorphic to $\SO_0(1,n+1)$. The component group is given by $G/G_0 \cong \Z / 2\Z \times \Z / 2\Z$. Hence there are four global characters $\chi_{\pm,\pm}$ of $G$ which are restricted to the subgroup $M=\Orm(1)\times \Orm(n)$ given by
$$\chi_{+,+}|_M=\mathbf{1}\otimes \mathbf{1}, \qquad \chi_{+,-}|_M=\mathbf{1}\otimes \det,$$
$$\chi_{-,+}|_M=\sgn\otimes \mathbf{1}, \qquad \chi_{-,-}|_M=\sgn\otimes \det.$$
We remark that $\chi_{+,-}$ is the determinant on $G$.

Note that as $\Orm(n)$-representations we have
$$\bigwedge\nolimits^p {(\C^n)} \cong \bigwedge\nolimits^{n-p} {(\C^n)} \otimes \det$$ such that for the principal series representation $\pi_{p,\lambda}^\pm$ we have
$$\chi_{+,-}\otimes \pi_{p,\lambda}^\pm \cong \pi_{n-p,\lambda}^\pm, \qquad \chi_{-,+}\otimes \pi_{p,\lambda}^\pm \cong \pi_{p,\lambda}^\mp,\qquad \chi_{-,-}\otimes \pi_{p,\lambda}^\pm \cong \pi_{n-p,\lambda}^\mp.$$

\subsection{Restriction to the identity component}\label{C:sec:restriction_identity_comp}

The following lemma is similar to \cite[Lemma 15.2]{kobayashi_speh_2018}.
\begin{lemma}\label{C:lemma:restriction_reducible_characters}
Let $\pi$ be an irreducible admissible representation of $G$.
\begin{enumerate}[label=(\roman{*})]
\item If $\chi \otimes \pi \not\cong \pi$ for all $\chi\in \{ \chi_{+,-},\chi_{-,+}, \chi_{-,-}  \}$ then $\pi|_{G_0}$ is irreducible.
\item  If $\chi_0 \otimes \pi \cong \pi$ for $\chi_0 \in  \{ \chi_{+,-},\chi_{-,+}, \chi_{-,-}  \}$ and $\chi \otimes \pi \not\cong \pi$ for all $\chi_0\neq \chi \in  \{ \chi_{+,-},\chi_{-,+}, \chi_{-,-}  \}$,
then $\pi|_{G_0}=\pi^{(+)}\oplus \pi^{(-)}$ decomposes into two non-isomorphic irreducible $G_0$ representations $\pi^{(+)}$ and $\pi^{(-)}$
\end{enumerate}
\end{lemma}

In the following we denote by $$\bar{\pi}_{p,\lambda}^\pm:= \pi_{p,\lambda}^\pm|_{G_0}$$ the restriction to the identity component.
We immediately obtain the following.

\begin{lemma}
\begin{enumerate}[label=(\roman{*})]
\item If $p \neq \frac{n}{2}$, the restriction $\bar{\pi}_{p,\lambda}^\pm$ is irreducible as $G_0$-representation if and only if $\pi_{p,\lambda}^\pm$ is irreducible as a $G$-representation. If $\pi_{p,\lambda}^\pm$ is reducible, the composition series of $\bar{\pi}_{p,\lambda}^\pm$ is given by the composition factors of $\pi_{p,\lambda}^\pm$ restricted to $G_0$.
\item If $p=\frac{n}{2}$, the restriction $\bar{\pi}_{p,\lambda}^\pm=\bar{\pi}_{p,\lambda}^{\pm(+)}\oplus \bar{\pi}_{p,\lambda}^{\pm(-)}$ is always reducible and decomposes into two non-isomorphic $G_0$-representations $\bar{\pi}_{p,\lambda}^{\pm(+)}$ and $\bar{\pi}_{p,\lambda}^{\pm(-)}$. The representations $\bar{\pi}_{p,\lambda}^{\pm(\pm)}$ are irreducible if and only if $\pi_{p,\lambda}^{\pm}$ is irreducible. If $\pi_{p,\lambda}^{\pm}$ is reducible, the composition series of $\bar{\pi}_{p,\lambda}^{\pm(\pm)}$ is given by the composition factors of $\pi_{p,\lambda}^{\pm}$ restricted to $G_0$, which are contained in $\bar{\pi}_{p,\lambda}^{\pm(\pm)}$.
\end{enumerate}
\end{lemma}

We use the corresponding notation $\bar{\tau}_{q,\nu}^\pm$ and $\bar{\tau}_{q,\nu}^{\pm(\pm)}$ for the components of the restriction to $G'_0\cong \SO_0(1,n)$.

\section{Composition series of \texorpdfstring{$\pi_{p,\lambda}^\pm$}{pi}}\label{C:sec:composition_series}
We recall results about the composition series of $\pi_{p,\lambda}^\pm$ and give explicit realizations as kernels of the standard Knapp--Stein intertwining operators.
\subsection{Irreducibility of principal series representations}

\begin{theorem}[\cite{kobayashi_speh_2018} Theorem~2.18]
$\pi_{p,\lambda}^\pm$ is reducible if and only if  $$\lambda\in (-\rho-1-\Z_{\geq 0}) \cup (\rho+1+\Z_{\geq 0})\cup\{ \rho-p, -\rho+p\}.$$
\end{theorem}

\subsection{The Knapp--Stein intertwining operator}
The composition series is closely connected to the Knapp--Stein intertwining operators which we introduce in this section.
Following \cite[Chapter 8]{kobayashi_speh_2018} we define the (normalized) Knapp--Stein intertwining operator
$T_{p,\lambda}$ as an element of $\mathcal{D}'(\R^n)\otimes \End_\C(\xi)$ by
$$T_{p,\lambda}=
	\begin{cases}
		\frac{1}{\Gamma(\lambda)}\abs{X}^{2(\lambda-\rho)}\sigma_p(\psi_n(X)) & 					\text{if $p\neq \frac{n}{2}$,} \\
		\frac{1}{\Gamma(\lambda+1)}\abs{X}^{2(\lambda-\rho)}\sigma_p(\psi_n(X)) & 				\text{if $p= \frac{n}{2}$.}
	\end{cases}$$
This defines a non-vanishing holomorphic family of intertwining operators $$T_{p,\lambda}: \pi_{p,\lambda}^\pm \to \pi_{p,-\lambda}^\pm.$$

The composition series of $\pi_{p,\lambda}^\pm$ is given in the following proposition which is \cite[Proposition~2.18 and Proposition~8.17]{kobayashi_speh_2018}

\begin{prop}
If  $\pi_{p,\lambda}^\pm$ is reducible it has composition series of length two and if additionally $\lambda \neq 0$ it has a unique irreducible submodule given by $\ker T_{p,\lambda}$. If $\lambda=0$ and $\pi_{p,\lambda}^\pm$ is reducible (hence $p=\frac{n}{2}$) it decomposes into the direct sum of two irreducible representations which are given by $$\ker\left(T_{p,0}\pm \frac{\pi^p}{p!} \operatorname{id}\right).$$
\end{prop}
Similarly we denote by $T'_{q,\nu}$ the Knapp--Stein intertwining operator for the subgroup $G'$, similarly defined and normalized.

\section{Unitary representations in the principal series of \texorpdfstring{$\Orm(1,n+1)$}{O(1,n+1)}}\label{C:sec:unitary_composition}
\label{C:sec:unitary_reps}
The principal series representation $\pi_{p,\lambda}^\pm$ is the smooth vectors of a tempered unitary representation if and only if $\lambda\in i\R$ and hence unitarizable (unitary principal series) and we denote its unitary closure by $\hat{\pi}_{p,\lambda}^\pm$. But for real parameters the principal series representation might be non-tempered unitarizable and irreducible (complementary series) or contain unitarizible composition factors which might be tempered or not.
\subsection{Criterion for unitarizability}
\begin{lemma}[See \cite{kobayashi_speh_2018} Example 3.32]
$\pi_{p,\lambda}^\pm$ is a complementary series representation if and only if
$$\lambda \in (-\abs{\rho-p},\abs{\rho-p}).$$
\end{lemma}
The following Proposition can easily be deduced using the calculations in \cite[Chapter~3.a]{BOO_1996}.
\begin{prop}
Let $p\neq 0,\frac{n}{2},n$.
$\pi_{p,\lambda}^\pm$ contains a unitarizable composition factor if and only if $\lambda \in \{ \rho-p,p-\rho\}$. In this case both the unique submodule as well as the unique quotient are unitarizable.
If $p=\frac{n}{2}$, $\pi_{p,\lambda}^\pm$ contains a unitarizable composition factor if and only if $\lambda=0$. In this case both submodules are unitarizable.
If $p=0,n$, $\pi_{p,\lambda}^\pm$ contains a unitarizable composition factor if and only if $\lambda\in \{ \pm(\rho+j), \, j\in \Z_{\geq 0} \}$. If $\lambda=-\rho-j$, the quotient $\pi_{p,\lambda}^\pm/ \ker T_{p,\lambda}$ is unitarizable and if $\lambda=\rho+j$ the submodule $\ker T_{p,\lambda}$ is unitarizable. Only in the special case $\lambda=\pm\rho$, also the other composition factor is unitarizable and one dimensional in this case.
\end{prop}

\subsection{Composition factors with non-trivial \texorpdfstring{$(\mathfrak{g},K)$}{(g,K)}-cohomology}
We introduce notation for the unitarizible composition factors. For $p\neq \frac{n}{2}$, $0< p<n$
let $\Pi_{p,\pm}$ be the unique proper submodule of $\pi_{p,p-\rho}^\pm$ and for $p=\frac{n}{2}$ let $$\Pi_{\frac{n}{2},\pm}:=\ker\left(T_{\frac{n}{2},0}- \frac{\pi^\frac{n}{2}}{(\frac{n}{2})!} \operatorname{id}\right)\subseteq \pi_{\frac{n}{2},0}^\pm.$$
Moreover let $$\Pi_{0,+}:=\chi_{+,+}, \qquad \Pi_{0,-}:=\chi_{-,+}, \qquad \Pi_{n+1,+}:=\chi_{+,-}, \qquad \Pi_{n+1,-}:=\chi_{-,-}.
$$
These four one-dimensional representations correspond to the unique finite dimensional unitarizable composition factors for $p=0,n$.

The following results are all due to \cite[Theorem~2.20]{kobayashi_speh_2018}
\begin{theorem}
\begin{enumerate}[label=(\roman{*}), ref=\thetheorem(\roman{*})]
\item  \label{C:theorem:triv_inf_character:seq}
For $0 \leq p\leq n$ we have the following exact sequences of $G$-modules.
$$0 \to\Pi_{p,\pm} \to \pi_{p,p-\rho}^\pm\to \Pi_{p+1,\mp}\to 0,$$
$$0 \to\Pi_{p+1,\mp} \to \pi_{p,\rho-p}^\pm\to \Pi_{p,\pm}\to 0.$$
These sequences split if and only if $p=\frac{n}{2}$.

\item The set $\{\Pi_{p,\alpha}, \, 0\leq p\leq n+1, \, \alpha=+,-\}$ characterizes exactly all irreducible smooth admissible $G$-representations whose infinitesimal character coincides with the infinitesimal character of the trivial representation (\textit{trivial infinitesimal character}).

\item All irreducible and unitarizable $(\mathfrak{g},K)$-modules with non-trivial $(\mathfrak{g},K)$-cohomology are exactly given by the underlying $(\mathfrak{g},K)$-modules of the elements of $\{\Pi_{p,\alpha}, \, 0\leq p\leq n+1, \, \alpha=+,-\}$.

\item The set of irreducible tempered representations of $G$ with trivial infinitesimal character is given for $n$ even by
$$\{ \Pi_{\frac{n}{2},\alpha}, \Pi_{\frac{n}{2}+1,\alpha}, \, \alpha=+,- \}$$
and for $n$ odd by
$$\{ \Pi_{\frac{n+1}{2},\alpha},\, \alpha=+,- \}.$$
In the odd case both representations are discrete series representations.
\item \label{C:theorem:triv_inf_character:isos}
We have the following isomorphisms of $G$-modules
$$
\chi_{+,-}\otimes\Pi_{p,\pm} \cong \Pi_{n+1-p,\mp}, \qquad \chi_{-,+}\otimes\Pi_{p,\pm} \cong \Pi_{p,\mp}, \qquad \chi_{-,-}\otimes\Pi_{p,\pm} \cong \Pi_{n+1-p,\pm}.
$$
\end{enumerate}
\end{theorem}

By Lemma~\ref{C:lemma:restriction_reducible_characters} and Theorem~\ref{C:theorem:triv_inf_character:isos} we immediately obtain the following.
\begin{corollary}
The restriction $\overline{\Pi}_{p,\pm}:=\Pi_{p,\pm}|_{G_0}$ is reducible if and only if $p=\frac{n+1}{2}$. In this case $\overline{\Pi}_{\frac{n+1}{2},\pm}$ decomposes as
$$\overline{\Pi}_{\frac{n+1}{2},\pm}=\overline{\Pi}_{\frac{n+1}{2},\pm}^{(+)}\oplus \overline{\Pi}_{\frac{n+1}{2},\pm}^{(-)}$$
into two non-isomorphic $G_0$-representations.
\end{corollary}
The restriction to the identity component becomes reducible if and only if $\Pi_{p,\pm}$ is a discrete series representation. In this case clearly both $\overline{\Pi}_{\frac{n+1}{2},\pm}^{(+)}$ and $\overline{\Pi}_{\frac{n+1}{2},\pm}^{(-)}$ are discrete series representations of $G_0$ and are contained in $\bar{\pi}_{\frac{n+1}{2},\frac{1}{2}}^\pm$ as submodules.

In the following we adapt the notation for the subgroup $G'$ and denote the representations with non-trivial $(\mathfrak{g}',K')$-cohomology by $\Pi'_{q,\pm}$.

\subsection{The additional cases for $p=0,n$}
We recall some facts about the infinite dimensional unitarizable composition factors in the cases $p=0,n$. The standard reference here is \cite{Johnson_Wallach_1977}.
For $\lambda=\rho+j$, $j\in \Z_{> 0}$ we denote by $I_{p,j,\pm}$ the unique proper submodule of $\pi_{p,\rho+j}^\pm$ and by $F_{p,j,\pm}$ the unique proper submodule of $\pi_{p,-\rho-j}^\pm$ which is finite dimensional. Then we have the following non-splitting short exact sequences of $G$-modules.
$$0 \to I_{p,j,\pm} \to \pi_{p,\rho+j}^\pm \to F_{p,j,\pm} \to 0,$$
$$0 \to F_{p,j,\pm} \to \pi_{p,-\rho-j}^\pm \to I_{p,j,\pm} \to 0.$$
Similarly we use the notation $I'_{q,j,\pm}$ and $F'_{q,j,\pm}$ for the composition factors of the $G'$-representations with $q=0,n-1$.

\subsection{Inner products on the complementary series and on unitarizable quotients}
Combining the results of the $K$-spectrum of the Knapp-Stein operator in \cite[Chapter~8.3.2]{kobayashi_speh_2018} with the relations between the scalars acting on $K$-types in \cite[Chapter~3.a]{BOO_1996} we obtain the following.
\begin{prop}
Let $\lambda \in [\rho-p,p-\rho] \setminus \{0\}$.
If $p<\frac{n}{2}$ the Knapp-Stein operator $T_{p, \lambda}$ acts by non-negative scalars on all $K$-types in $\pi_{p,\lambda}^\pm$ and if $p>\frac{n}{2}$ it acts by non-positive scalars on all $K$-types in $\pi_{p,\lambda}^\pm$. If the Knapp--Stein operator vanishes on a $K$-type it is contained in the submodule $\ker T_{p,\lambda}$. 
For $p=0$ and $\lambda \in (-\rho-1-\Z_{\geq 0})$,the Knapp-Stein operator $T_{0, \lambda}$ acts by non-negative scalars on all $K$-types in $\pi_{0,\lambda}^\pm$ and if $p=n$ and $\lambda \in (-\rho-1-\Z_{\geq 0})$,the Knapp-Stein operator $T_{n, \lambda}$ acts by non-positive scalars on all $K$-types in $\pi_{n,\lambda}^\pm$.
\end{prop}
In the case $p=\frac{n}{2}$ the only unitarizable composition factors occur at $\lambda=0$ which is already in the unitary principal series and there is no complementary series.
By the proposition above we define the following pairing which is an inner product on the complementary series.
$$\langle \cdot, \cdot \rangle_{p,\lambda}:=
\begin{cases}
\langle \cdot, {T}_{p,\lambda} \cdot \rangle_{L^2(K)} & \text{if $p< \frac{n}{2}$,} \\
-\langle \cdot, {T}_{p,\lambda} \cdot \rangle_{L^2(K)} & \text{if $p > \frac{n}{2}$.} \\
\end{cases}
$$
We denote the corresponding unitary closures by $\hat{\pi}_{p,\lambda}^\pm$.
Let $\lambda \in (-\rho-1-\Z_{\geq 0}) \cup( \{ \rho-p,p-\rho\} \setminus \{0\})$ such that $\pi_{p,\lambda}^\pm$ contains a unitarizable quotient and let $\pr_\lambda$ be the projection. Since $\ker T_{p,\lambda}$ is the unique proper submodule, we obtain an induced intertwiner $$T^{quo}_{p,\lambda}: \pi_{p,\lambda}^\pm / \ker T_{p,\lambda} \to \pi_{p,-\lambda}^\pm$$
which is an isomorphism onto the unique proper submodule of $\pi_{p,-\lambda}^\pm$ and which is essentially the Knapp-Stein operator. Then we similarly define the following inner product on the quotients.
$$\langle \cdot, \cdot \rangle_{p,\lambda,quo}:=
\begin{cases}
\langle \cdot,T^{quo}_{p,\lambda} \cdot \rangle_{L^2(K)} & \text{if $p< \frac{n}{2}$,} \\
-\langle \cdot,T^{quo}_{p,\lambda} \cdot \rangle_{L^2(K)} & \text{if $p > \frac{n}{2}$.} \\
\end{cases}
$$
We remark that by construction for $f \in \pi_{p,\lambda}^\pm$ we have
$$\langle \pr_\lambda f, \pr_\lambda f\rangle_{p,\lambda,quo}=\langle f, f \rangle_{p,\lambda}.$$
We denote the corresponding unitary closures by $\widehat{\Pi}_{p,\pm}$ resp. $\hat{I}_{p,j,\pm}$.
In the following we use the notation $T'^{quo}_{q,\nu}$ similary for $G'$, $\langle \cdot, \cdot \rangle_{q,\nu}, \langle \cdot, \cdot \rangle_{q,\nu,quo}$ for the inner products and $\widehat{\Pi}'_{q,\pm}$ and $\hat{I}'_{q,j,\pm}$ for the corresponding unitary closures.

\section{Classification of symmetry breaking operators}\label{C:sec:classification_sbos}
We recall main result of \cite{kobayashi_speh_2018}. By \cite[Theorem~1.5]{kobayashi_speh_2018} we have that 
$$\Hom_{G'}(\pi_{p,\lambda}^\pm|_{G'},\tau_{q,\nu}^\pm) \neq  \{0\}
\Rightarrow q \in \{p-2,p-1,p,p+1 \}.
$$
We will restrict ourselves to the cases $q=p-1,p$, since the two cases are enough for our purpose of decomposing unitary representations in $\pi_{p,\lambda}^\pm$.

For $\Re(\nu)\ll 0$ and $\Re(\lambda)+\Re(\nu) \gg 0$ consider the $\Hom_\C(\sigma_p,\delta_q)$-valued distribution kernel
$$u_{(p,\lambda),(q,\nu)}^+(X):=\abs{X}^{-2(\nu+\rho')}\abs{X_n}^{\lambda-\rho+\nu+\rho'}\pr_{\sigma_p\to\delta_q}(\sigma_p(\psi_n(X)))$$
on $\R^n\cong \bar{N}$.
By \cite[Theorem~3.10]{kobayashi_speh_2018}, $u_{(p,\lambda),(q,\nu)}^+$ defines a symmetry breaking operator
$$A_{(p,\lambda),(q,\nu)}^+\in \Hom_{G'}(\pi_{p,\lambda}^\pm|_{G'},\tau_{q,\nu}^\pm)$$
in the sense of Theorem~\ref{thm:KS-distr-kernel}.
%
%

Similarly for $\Re(\nu)\ll 0$ and $\Re(\lambda)+\Re(\nu) \gg 0$ the $\Hom_\C(\sigma_p,\delta_q)$-valued distribution kernel
$$
u_{(p,\lambda),(q,\nu)}^-(X):=\abs{X}^{-2(\nu+\rho')}\abs{X_n}^{\lambda-\rho+\nu+\rho'}\sgn(X_n)\pr_{\sigma_p\to\delta_q}(\sigma_p(\psi_n(X))),
$$
defines a symmetry breaking operator $$A_{(p,\lambda),(q,\nu)}^-\in \Hom_{G'}(\pi_{p,\lambda}^\pm|_{G'},\tau_{q,\nu}^\mp).$$
We define the renormalizations
$$\tilde{A}_{(p,\lambda),(q,\nu)}^+:=\frac{1}{\Gamma(\frac{\lambda+\rho+\nu-\rho'}{2})\Gamma(\frac{\lambda+\rho-\nu-\rho'}{2})}A_{(p,\lambda),(q,\nu)}^+$$
and
$$\tilde{A}_{(p,\lambda),(q,\nu)}^-:=\frac{1}{\Gamma(\frac{\lambda+\rho+\nu-\rho'+1}{2})\Gamma(\frac{\lambda+\rho-\nu-\rho'+1}{2})}A_{(p,\lambda),(q,\nu)}^-.$$
Then $\tilde{A}_{(p,\lambda),(q,\nu)}^\pm$ define families of symmetry breaking operators which extend holomorphically in $\lambda,\nu\in \C$ (see \cite[Theorem~3.10]{kobayashi_speh_2018}). In this sense we can consider $A^\pm_{(p,\lambda),(q,\nu)}$ meromorphically extended to the whole plane $(\lambda,\nu)\in \C^2$, with their meromorphic behavior encoded in the normalizing Gamma-factors. In particular we can consider the distribution kernels $u_{(p,\lambda),(q,\nu)}^\pm$ meromorphically extended to $(\lambda,\nu)\in \C^2$. According to \eqref{eq:kernel_operator} we explicitly have the following formulas for the action of $A^\pm_{(p,\lambda),(q,\nu)}$.
\begin{prop}\label{prop:A_action}
	For $f\in \pi_{p,\lambda}^\pm$ and $h\in G'$ we have
	$$	\left(A_{(p,\lambda),(q,\nu)}^\pm f\right)(h)=\int_{\R^n}u_{(p,\lambda),(q,\nu)}^\pm(X)f(h\overline{n}_X)\, dX.$$
\end{prop}

We define the following subsets of $\C^2$ for $\alpha\in \{+,-\}$.
$$L^\alpha:=\{(-\rho-j,-\rho'-i), \, i,j \in \Z, \, 0\leq j \leq i\text{ and } i +\frac{1-(\alpha1)}{2} \equiv j \mod 2 \},$$
$$L(p,q)^\alpha:=  
\begin{cases}
L^\alpha & \text{if $p=q=0$ or $p=q+1=n$,} \\
(L^+ \setminus \{\nu=-\rho'\}) \cup \{(p-\rho,q-\rho')\}  & \text{if $1\leq p < n$, $q=p$ and $\alpha=+$}, \\
(L^- \setminus \{\nu=-\rho'\})  & \text{if $1\leq p < n$, $q=p-1,p$ and $\alpha=-$}, \\
(L^+ \setminus \{\nu=-\rho'\}) \cup \{(\rho-p,\rho'-q)\}  & \text{if $1\leq p < n$, $q=p-1$ and $\alpha=+$}.
\end{cases}
$$
\begin{theorem}[\cite{kobayashi_speh_2018} Theorem~3.19]
$\tilde{A}_{(p,\lambda),(q,\nu)}^\pm=0$ if and only if $(\lambda,\nu) \in L(p,q)^\pm$.
\end{theorem}

Moreover we define two renormalizations of $A_{(p,\lambda),(q,\nu)}^\pm$.
Fix $\nu$ such that there exists a $\mu\in \C$ so that $(\mu,\nu)\in L(p,q)^\pm$. We define
$$\tilde{\tilde{A}}_{(p,\lambda),(q,\nu)}^+:=\Gamma\left(\frac{\lambda+\rho-\nu-\rho'}{2}\right)\tilde{A}_{(p,\lambda),(q,\nu)}^+,$$
$$\tilde{\tilde{A}}_{(p,\lambda),(q,\nu)}^-:=\Gamma\left(\frac{\lambda+\rho-\nu-\rho'+1}{2}\right)\tilde{A}_{(p,\lambda),(q,\nu)}^-.$$
Then $\tilde{\tilde{A}}_{(p,\lambda),(q,\nu)}^\pm$ defines a non-vanishing family of symmetry breaking operators which is holomoprhic in $\lambda$.

For fixed $\lambda+\rho-\nu-\rho' \in -\Z_{\geq 0}$ and $q=p-1,p$ we define the meromorphic functions
$$c_C(p,q,\nu):=\Gamma(\nu+\rho'+1)
\times
\begin{cases}
1 & \text{if $p\neq 0,n$ and $\lambda-\nu\neq -\frac{1}{2}$,} \\
\frac{1}{\nu+\rho'-p} & \text{if $p=0$ or $\lambda-\nu=-\frac{1}{2}$ and $q=p$,} \\
\frac{1}{\nu-\rho'+p-1} & \text{if $p=n$ or $\lambda-\nu=-\frac{1}{2}$ and $q=p-1$}
\end{cases}$$
and we define the operators
$$C_{(p,\lambda),(q,\nu)}^+:=c_C(p,q,\nu)\tilde{A}_{(p,\lambda),(q,\nu)}^+,$$
for $\lambda+\rho-\nu-\rho' \in -2\Z_{\geq 0}$ and
$$C_{(p,\lambda),(q,\nu)}^-:=c_C(p,q,\nu)\tilde{A}_{(p,\lambda),(q,\nu)}^-,$$
for $\lambda+\rho-\nu-\rho' \in -1-2\Z_{\geq 0}$.
Then $C_{(p,\lambda),(q,\nu)}^\pm$ defines a non-vanishing family of symmetry breaking operators which is holomoprhic in $\nu$.

We remark that we set all symmetry breaking operators for $p=0$ and $q=p-1$ as well as $p=n$ and $q=p$ to zero. That way we can prove many results in the following in a uniform way.

\begin{theorem}[Classification of symmetry breaking operators, see \cite{kobayashi_speh_2018} Theorem 3.19 and Theorem 3.26]
For $(\lambda,\nu)\notin L(p,q)^+$ we have
 $$ \Hom_{G'}(\pi_{p,\lambda}^\pm|_{G'},\tau_{q,\nu}^\pm)=\C \tilde{A}_{(p,\lambda),(q,\nu)}^+$$ 
and for $(\lambda,\nu)\in L(p,q)^+$ we have
$$\Hom_{G'}(\pi_{p,\lambda}^\pm|_{G'},\tau_{q,\nu}^\pm)=\C \tilde{\tilde{A}}_{(p,\lambda),(q,\nu)}^+\oplus \C C_{(p,\lambda),(q,\nu)}^+.$$
For $(\lambda,\nu)\notin L(p,q)^-$ we have
$$ \Hom_{G'}(\pi_{p,\lambda}^\pm|_{G'},\tau_{q,\nu}^\mp)=\C \tilde{A}_{(p,\lambda),(q,\nu)}^-$$ 
and for $(\lambda,\nu)\in L(p,q)^-$ we have
$$\Hom_{G'}(\pi_{p,\lambda}^\pm|_{G'},\tau_{q,\nu}^\mp)=\C \tilde{\tilde{A}}_{(p,\lambda),(q,\nu)}^-\oplus \C C_{(p,\lambda),(q,\nu)}^-.$$
\end{theorem}

\section{Functional equations}\label{C:sec:func_equations}
We recall the following functional equations for the symmetry breaking operators and the Knapp--Stein intertwiners.
\begin{theorem}[\cite{kobayashi_speh_2018} Theorem~9.24, Theorem~9.25 and Theorem~9.31]
\label{C:theorem:functional_equations}
\begin{align*}
T_{p-1,\nu}'\circ \tilde{A}_{(p,\lambda),(p-1,\nu)}^\pm&=-\frac{\pi^{\frac{n-1}{2}}}{\Gamma(\nu+\rho'+1)}\tilde{A}_{(p,\lambda),(p-1,-\nu)}^\pm &&\times
\begin{cases}
(\nu-\rho'+p-1) & \text{if $p\neq \frac{n+1}{2}$,} \\
1 &\text{if $p= \frac{n+1}{2}$.} 
\end{cases}\\
T_{p,\nu}'\circ \tilde{A}_{(p,\lambda),(p,\nu)}^\pm&=\frac{\pi^{\frac{n-1}{2}}}{\Gamma(\nu+\rho'+1)}\tilde{A}_{(p,\lambda),(p,-\nu)}^\pm &&\times
\begin{cases}
(\nu+\rho'-p) & \text{if $p\neq \frac{n-1}{2}$,} \\
1 &\text{if $p= \frac{n-1}{2}$.} 
\end{cases}
\\
\tilde{A}_{(p,-\lambda),(p-1,\nu)}^\pm\circ T_{p,\lambda}&=\frac{\pi^{\frac{n}{2}}}{\Gamma(-\lambda+\rho+1)}\tilde{A}_{(p,\lambda),(p-1,\nu)}^\pm &&\times
\begin{cases}
(\lambda+\rho-p) & \text{if $p\neq \frac{n}{2}$,} \\
1 &\text{if $p= \frac{n}{2}$.} 
\end{cases}\\
\tilde{A}_{(p,-\lambda),(p,\nu)}^\pm\circ T_{p,\lambda}&=-\frac{\pi^{\frac{n}{2}}}{\Gamma(-\lambda+\rho+1)}\tilde{A}_{(p,\lambda),(p,\nu)}^\pm &&\times
\begin{cases}
(\lambda-\rho+p) & \text{if $p\neq \frac{n}{2}$,} \\
1 &\text{if $p= \frac{n}{2}$.} 
\end{cases}
\end{align*}
For the case $\nu=\frac{1}{2}$ and $p=\frac{n}{2}$ we have
$$\tilde{\tilde{A}}_{(p,0),(p,\nu)}^+\circ T_{p,0}=\frac{\pi^\frac{n}{2}}{(\frac{n}{2})!}\tilde{\tilde{A}}_{(p,0),(p,\nu)}^+,
$$
$$\tilde{\tilde{A}}_{(p,0),(p-1,\nu)}^+\circ T_{p,0}=-\frac{\pi^\frac{n}{2}}{(\frac{n}{2})!}\tilde{\tilde{A}}_{(p,0),(p,\nu)}^+,
$$
and for general $\nu$ such that $\tilde{\tilde{A}}_{(p,\lambda),(q,\nu)}$ exists we have for $p\neq \frac{n}{2}$
	$$
	T_{q,\nu}'\circ \tilde{\tilde{A}}_{(p,\lambda),(q,\nu)}^\pm=0.
	$$
\end{theorem}
We remark that the last functional equation is not contained in \cite{kobayashi_speh_2018} but is proven in the same way as the one before in \cite[Theorem~9.28]{kobayashi_speh_2018}.
By the theorem above we define for $q=p,p-1$ the meromorphic functions $t'(p,q,\nu)$ and $t(p,q,\lambda)$ such that
$$T_{q,\nu}'\circ \tilde{A}_{(p,\lambda),(q,\nu)}^\pm=t'(p,q,\nu)\tilde{A}_{(p,\lambda),(q,-\nu)}^\pm$$
and 
$$
\tilde{A}_{(p,-\lambda),(q,\nu)}^\pm\circ T_{p,\lambda}=t(p,q,\lambda)\tilde{A}_{(p,\lambda),(q,\nu)}^\pm.$$

\section{Symmetry breaking operators into quotients}\label{C:sec:sbo_quotients}
Let $\nu\in \R\setminus \{0\}$ such that $\tau_{q,\nu}^\pm$ has an unique non-trivial quotient, i.e. $\nu\in (-\rho'-1-\Z_{\geq 0}) \cup (\rho'+1+\Z_{\geq 0}) \cup \{ \rho'-q,q-\rho'\}\setminus \{0\}$ and let $$\pr_\nu: \tau_{q,\nu}\to \tau_{q,\nu}/\ker T_{q,\nu}$$ be the projection. Then for a smooth admissible $G$-representation $\pi$, clearly an element $A\in \Hom_{G'}(\pi|_{G'}, \tau_{q,\nu}^\pm)$ defines by composition with the projection an element of $A^{quo}\in  \Hom_{G'}(\pi|_{G'}, \tau_{q,\nu}^\pm/ \ker T'_{q,\nu})$. Since the unique submodule of $\tau_{q,\nu}^\pm$ is $\ker T'_{q,\nu}$, the Knapp-Stein operator $T'_{q,\nu}$ induces an intertwiner
$$T'^{quo}_{q,\nu}:  \tau_{q,\nu}^\pm/ \ker T_{q,\nu}' \to \tau_{q,-\nu}^\pm,$$
which is an isomorphism onto $\operatorname{im} T'_{q,\nu}=\ker T'_{q,-\nu}$, which is the unique proper submodule of $\tau_{q,-\nu}^\pm.$

\begin{prop}\label{C:prop:func_equations_quotients}
Let
$\nu\in (-\rho'-1-\Z_{\geq 0}) \cup (\rho'+1+\Z_{\geq 0}) \cup \{ \rho'-q,q-\rho'\}\setminus \{0\}$ and $q=p-1,p$ .
\begin{enumerate}[label=(\roman{*})]
\item For the operator $\tilde{A}_{(p,\lambda),(q,\nu)}^{quo}$ the following functional equations holds.
$$
T'^{quo}_{q,\nu} \circ\tilde{A}_{(p,\lambda),(q,\nu)}^{\pm,quo} =t'(p,q,\nu) \tilde{A}_{(p,\lambda),(q,-\nu)}^\pm
$$
\item For $(\lambda_0,-\nu)\in L^\pm(p,q)$ the renormalized operator
$$\tilde{\tilde{\tilde{A}}}_{(p,\lambda),(q,\nu)}^{\pm, quo}:=\lim_{\lambda \to \lambda_0}\left( \Gamma\left(\frac{\lambda+\rho+\nu-\rho'}{2}-\frac{\pm1 -1}{4}\right) \pr_\nu \circ \tilde{A}_{(p,\lambda),(q,\nu)}^\pm \right)$$
is a well defined symmetry breaking operator and satisfies the functional equation
$$T'^{quo}_{q,\nu} \circ\tilde{\tilde{\tilde{A}}}_{(p,\lambda),(q,\nu)}^{\pm, quo} = t'(p,q,\nu) \tilde{\tilde{A}}_{(p,\lambda_0),(q,-\nu)}^\pm.$$
\end{enumerate}
\end{prop}
\begin{proof}
(i) follows immediately from the functional equations of Theorem~\ref{C:theorem:functional_equations}.

Ad (ii). Again by Theorem~\ref{C:theorem:functional_equations} we have 
$$t'(p,q,\nu) \tilde{\tilde{A}}_{(p,\lambda_0),(q,-\nu)}^\pm= \lim_{\lambda \to \lambda_0}\left(t'(p,q,\nu) 
\Gamma\left(\frac{\lambda+\rho+\nu-\rho'}{2}-\frac{\pm1 -1}{4}\right)  \circ {\tilde{A}}_{(p,\lambda),(q,-\nu)}^\pm
\right).$$
Then the statement follows from the application of (i).
\end{proof}

We consider the special case $\nu=-\frac{1}{2}$ and $p=\frac{n}{2}$. In this case both $\tau_{p,\nu}^\pm$ and $\tau_{p-1,\nu}^\mp$ contain the discrete series representation $\Pi'_{p,\pm}$ as a quotient. 
\begin{prop}\label{C:prop:norms_equal}
For $\lambda \in i\R$ and $f\in \pi_{\frac{n}{2},\lambda}^\pm$
$$\frac{\abs{\lambda}^2}{4}\norm{{\tilde{A}}_{(\frac{n}{2},\lambda),(\frac{n}{2},-\frac{1}{2})}^{+,quo}f}^2_{\frac{n}{2},-\frac{1}{2},quo}=\norm{{\tilde{A}}_{(\frac{n}{2},\lambda),(\frac{n}{2}-1,-\frac{1}{2})}^{-,quo}f}^2_{\frac{n}{2}-1,-\frac{1}{2},quo},$$
$$\frac{\abs{\lambda}^2}{4}\norm{\tilde{A}_{(\frac{n}{2},\lambda),(\frac{n}{2}-1,-\frac{1}{2})}^{+,quo}f}^2_{\frac{n}{2},-\frac{1}{2},quo}=\norm{{\tilde{A}}_{(\frac{n}{2},\lambda),(\frac{n}{2},-\frac{1}{2})}^{-,quo}f}^2_{\frac{n}{2}-1,-\frac{1}{2},quo},$$
and for $\lambda=0$ and $f\in \pi_{\frac{n}{2},0}^\pm$
$$\norm{{\tilde{\tilde{\tilde{A}}}}_{(\frac{n}{2},\lambda),(\frac{n}{2},-\frac{1}{2})}^{+,quo}f}^2_{\frac{n}{2},-\frac{1}{2},quo}=\norm{{\tilde{A}}_{(\frac{n}{2},\lambda),(\frac{n}{2}-1,-\frac{1}{2})}^{-,quo} f}^2_{\frac{n}{2}-1,-\frac{1}{2},quo},$$
$$\norm{{\tilde{\tilde{\tilde{A}}}}_{(\frac{n}{2},\lambda),(\frac{n}{2}-1,-\frac{1}{2})}^{+,quo}}^2_{\frac{n}{2},-\frac{1}{2},quo}=\norm{{\tilde{A}}_{(\frac{n}{2},\lambda),(\frac{n}{2},-\frac{1}{2})}^{-,quo}f}^2_{\frac{n}{2}-1,-\frac{1}{2},quo}.$$
\end{prop}
\begin{proof}
Since $\chi_{-,-}\otimes \Pi'_{p,\pm}=\Pi'_{p,\pm}$, the map
$$\otimes \chi_{-,-}: \tau_{p,\nu}^\pm \to \tau_{p-1,\nu}^\mp$$
induces an isomorphism of the quotients $\tau_{p,\nu}^\pm/ \ker T'_{p,\nu}$ and $\tau_{p-1,\nu}^\mp / \ker T'_{p-1,\nu}$ which are both isomorphic to $\Pi'_{\frac{n}{2},\pm}$. Then by Theorem~\ref{C:theorem:mult_one} composition with this isomorphism yields an isomorphism
$$\Hom_{G'}(\pi|_{G'}, \tau_{p,\nu}^\pm/ \ker T'_{p,\nu}) \to \Hom_{G'}(\pi|_{G'}, \tau_{p-1,\nu}^\mp / \ker T'_{p-1,\nu})$$ for each irreducible $G$-representation $\pi$. Hence for $\lambda \in i\R \setminus \{ 0\}$
$$\norm{{\tilde{A}}_{(\frac{n}{2},\lambda),(\frac{n}{2},-\frac{1}{2})}^{+,quo}f}^2_{\frac{n}{2},-\frac{1}{2},quo} \text{ and }
\norm{{\tilde{A}}_{(\frac{n}{2},\lambda),(\frac{n}{2}-1,-\frac{1}{2})}^{-,quo}f}^2_{\frac{n}{2}-1,-\frac{1}{2},quo}$$ as well as
$$\norm{{\tilde{A}}_{(\frac{n}{2},\lambda),(\frac{n}{2}-1,-\frac{1}{2})}^{+,quo}f}^2_{\frac{n}{2},-\frac{1}{2},quo} \text{ and }
  \norm{{\tilde{A}}_{(\frac{n}{2},\lambda),(\frac{n}{2},-\frac{1}{2})}^{-,quo} f}^2_{\frac{n}{2}-1,-\frac{1}{2},quo}$$ must be constant positive scalar multiples of each other for every $f$. Then it is enough to check the eigenvalues of the symmetry breaking operators in question on a $K$-type contained in the quotient which yields the result for $\lambda\neq 0$ by \cite[Theorem~9.8]{kobayashi_speh_2018}. For $\lambda =0$ we take the limit $\lambda \to 0$ to obtain the result by Proposition~\ref{C:prop:func_equations_quotients}.
\end{proof}

\section{Structure of the open orbit as homogeneous \texorpdfstring{$G'$}{G'}-space}\label{C:sec:homogeneous_G'_space}
The following Lemma is a key point in the decomposition of unitary representations to come. It reduces the problem of decomposing a unitary representation into a problem of harmonic analysis on a homogeneous $G'$-space.
\begin{lemma}\label{C:lemma:H-Stab}
	\begin{enumerate}[label=(\roman{*})]
		\item Let $\tilde{K}':=\operatorname{Stab}_{G'}(\overline{n}_{e_n}P)$. Then $\tilde{K}'=\Orm(n)$ and $K'/\tilde{K}'\cong \Orm(1)$.
		\item We have $G'\cdot\overline{n}_{e_n}P=\mathcal{O}_A$ is the open $P'$-orbit in $G/P$.
	\end{enumerate}
\end{lemma}

Lemma~\ref{C:lemma:H-Stab} implies that $\mathcal{O}_A\cong G'/\tilde{K'}=O(1,n)/O(n)$.

For proving this lemma we make use of the explicit action of $G'$ on $G/P\cong K/M$ and of the Iwasawa-decomposition of elements of $\Nbar$.
Therefore consider the map $$K/M \to S^{n} \subseteq \R^{n+1}$$ given by
$$k=\begin{pmatrix}a & & \\& b & \ast \\& c& \ast \end{pmatrix} \mapsto {\left(\frac{b}{a},\frac{c}{a}\right)}.
$$
\begin{lemma}
\begin{enumerate}[label=(\roman{*}), ref=\thetheorem(\roman{*})]
\item \label{C:lemma:compact_picture_sphere_map}The map $$K/M \to S^{n} \subseteq \R^{n+1}$$
$$k=\begin{pmatrix}a & & \\& b & \ast \\& c& \ast \end{pmatrix} \to {\left(\frac{b}{a},\frac{c}{a}\right)}
$$
 is a $G$ equivariant isomorphism with the action of $G$ on $S^{n}$ given by
 $$g\cdot \omega =\frac{({g}(1,\omega)^t)'}{({g}(1,\omega)^t)^1},$$
 where $(\, \cdot \,)^1$ is the first and $(\, \cdot \,)'$ the remainign coordinates of the vector. 
\item \label{C:lemma:KAN_decomp}
We have $\overline{n}_X=\kappa(\overline{n}_X)e^{H(\overline{n}_X)}n\in KAN$, with
$$\kappa(\overline{n}_X)=\begin{pmatrix}a & & \\& b & \ast \\& c& \ast \end{pmatrix},$$
with $$a=\frac{1+\abs{X}^2}{\sqrt{(1+\abs{X}^2)^2}},\qquad b=\frac{1-\abs{X}^2}{\sqrt{(1+\abs{X}^2)^2}}$$
and $$c=\frac{2X^T}{{(1+\abs{X}^2)}},$$
and $$H(\overline{n}_{X})=\log (1+\abs{X}^2)H.$$
\end{enumerate}
\end{lemma}
Using the $KAN$ decomposition of the Lemma above,we obtain a map
$$\nbar \to K/M \cong S^{n},$$  by multiplying with $\diag(a^{-1},a^{-1},\mathbf{1}_n) \in  M$ from the right:
\begin{equation}\label{C:eq:KAN_map}
\overline{n}_{X}\mapsto \left(\frac{1-\abs{X}^2}{1+\abs{X}^2},\frac{2X}{1+\abs{X}^2}   \right).
\end{equation}
\begin{proof}
This is easily checked by computing the corresponding matrix decomposition.
\end{proof}

\begin{proof}[Proof of Lemma~\ref{C:lemma:H-Stab}]
Ad (i): From Lemma~\ref{C:lemma:compact_picture_sphere_map} and \eqref{C:eq:KAN_map} it follows immediately that $K_0'=\Orm(n)$, embedded in $K'$ in the bottom right corner.

	Ad(ii): 
	By Proposition~\ref{C:prop:longest_weyl_group_element_action} we have $ \overline{n}_{e_n}P=\tilde{w}_0\overline{n}_{-e_n}P$ such that by Corollary~\ref{C:cor:orbits} we have $P'\cdot \overline{n}_{e_n}P=(\Nbar \setminus \Nbar')P$, since $w_0$ fixes $(\Nbar \setminus \Nbar')P$ again by Proposition~\ref{C:prop:longest_weyl_group_element_action}.
	By the Bruhat-decompositon we have $G'=P'\sqcup N'\tilde{w}_0P'$ and $N'\tilde{w}_0=\tilde{w}_0\Nbar'$ obviously fixes $(\Nbar \setminus \Nbar')P$.
\end{proof}

By Lemma~\ref{C:lemma:H-Stab} we can define an $G'$-equivariant map
given by $$\Phi: f\mapsto f|_{\mathcal{O}_A}(\; \cdot\; \overline{n}_{e_n}).$$
In fact this map is up to inner-automorphism onto the smooth sections
of the $G'$-bundle over $G'/\tilde{K}',$ corresponding to the representation $\bigwedge\nolimits^p(\C^n)$ of $\tilde{K}'=\Orm(n)$.
In the following let for $g\in G$, $g=\overline{n}e^{\overline{H}(g)} \overline{\kappa}(g)\in \Nbar A K$ be the $\Nbar A K$ Iwasawa-decomposition. 
We define $$w:=
\begin{pmatrix}
	& -\mathbf{1}_{n-1} \\
	-1 &
\end{pmatrix}\in \Orm(n).$$
Then it is easily verified, that for $k\in \Orm(n)$
\begin{equation}\label{eq:K_M_w_conjugation}
	\diag(1,k,1)\overline{n}_{e_n}=\overline{n}_{e_n}\diag(1,1,wkw^{-1}).
\end{equation}

\begin{lemma}\label{C:lemma:Phi}
The map $\Phi$ defines a linear continuous $G'$-equivariant map $$\pi_{p,\lambda}^\pm|_{G'}\to C^\infty\left(G'/\tilde{K}', \bigwedge\nolimits^p(\C^n)\right).$$
\end{lemma}
\begin{proof}
	This follows immediately from \eqref{eq:K_M_w_conjugation}.
\end{proof}
By Mackey theory, the restriction to an open subset carries enough information for our purpose.

	\begin{lemma}\label{C:lemma:L^2_condition}
		For $\Re(\lambda) > -\frac{1}{2}$, the map $\Phi$ extends to a $G'$-equivariant map
		$$\pi_{p,\lambda}^\pm|_{G'}\to L^2\left(G'/\tilde{K}',\bigwedge\nolimits^p(\C^n)\right)$$
		which is unitary for $\lambda \in i\R$.
	\end{lemma}
	Recall that by the Iwasawa decomposition the following integral formula holds
	\begin{equation}\label{C:eq:integral_formula_KAN}
	\int_{G'/K'}f(g)dg=\int_{ \Nbar'\times\mathfrak{a}} f(\overline{n}e^X) e^{2\rho'(X)}d\Nbar' dX.
	\end{equation}
	Moreover, 
	let $\omega=\pm1$ and $\Xi=\diag(\omega,\omega,\mathbf{1}_n)\in M'$. Then by Lemma~\ref{C:prop:matrix_decompositions:ii},
	\begin{equation}\label{C:eq:center_matrix_decomp}
	\Xi\overline{n}_{e_n}=\overline{n}_{\omega e_n}\Xi.
	\end{equation}

	\begin{proof}[Proof of Lemma~\ref{C:lemma:L^2_condition}]
		Let $f\in \pi_{p,\lambda}^\pm$. We choose representatives $\Xi$ of $K'/\tilde{K}'\cong \Orm(1)$. By \eqref{C:eq:integral_formula_KAN}
		\begin{align} \nonumber
		\norm{\Phi f}_{L^2(G'/\tilde{K}')}^2=&\int_{G'}\abs{f(g\overline{n}_{e_n})}^2dg\\=&\int_{ \Nbar' \times \mathfrak{a}\times \Orm(1)} \abs{f(\overline{n}_{(X',0)} e^{rH}\Xi\overline{n}_{e_n})}^2 e^{2\rho'r}dX'\;dr\; d \omega \nonumber  \\
		=&\int_{ \Nbar' \times \mathfrak{a}\times \Orm(1)} \abs{f(\overline{n}_{(X',e^{-r}\omega)})}^2 e^{-2r(\Re(\lambda)+\rho-\rho')}dX'\;dr\;d\omega \nonumber \\
		=&\int_{ \Nbar' \times \R_+ \times \Orm(1)} \abs{f(\overline{n}_{(X',t\omega)})}^2 t^{2(\Re(\lambda)+\rho-\rho')-1}d(X',Z)\;dt\;d\omega \nonumber\\ \label{C:eq:isometry}
		=&\int_{\Nbar} \abs{f(\overline{n}_{(X',X_n)})}^2 \abs{X_n}^{2\Re(\lambda)}dX.
		\end{align}
		Now as above we have $$\overline{n}_{(X',X_n)}=ke^{\log (1+\abs{X'}^2+\abs{X_n}^2)H}n \in KAN,$$ such that  there exists a non-negative constant $c_f$ such that $$\abs{f(\overline{n}_{(X',X_n)})}^2 \leq c_f ((1+\abs{X'}^2+\abs{X_n}^2)^2)^{-(\Re(\lambda)+\rho)}.$$
		Hence
		\begin{align*}
		\norm{\Phi f}_{L^2(G'/\tilde{K}')}^2\leq &c_f \int_{\Nbar}
		(1+\abs{X'}^2+\abs{X_n}^2)^{-2(\Re(\lambda)+\rho)}\abs{X_n}^{2\Re(\lambda)}\,dX \\
		=& \tilde{c}_f \int_{(\R_+)^2}	(1+r^2+s^2)^{-2(\Re(\lambda)+\rho)}r^{n-2}s^{2\Re(\lambda)}dr\;ds,
		\end{align*}
		where $\tilde{c}_f=2\operatorname{Vol}(S^{n-2})c_f$. 
		Using polar coordinates on $(\R_+)^2$ we find
		$$
		\norm{\Phi f}_{L^2(G'/\tilde{K}')}^2\leq \frac{\tilde{c}_f}{4} 
		 \int_{0}^{\frac{\pi}{2}}\cos^{n-2}{\phi} \sin^{2\Re(\lambda)}\phi \;d\phi \;
		\int_{0}^{\infty} x^{(\Re(\lambda)+\rho)-1}(1+x)^{-2(\Re(\lambda)+\rho)} \;dx,
		$$
		which converges for $\Re\lambda> -\frac{1}{2}$.
		That the map is a unitary one for the unitary principal series follows from the equation \eqref{C:eq:isometry}.
	\end{proof}

Clearly the bundle $C^\infty \left(G'/\tilde{K}', \bigwedge\nolimits^p(\C^n)\right)$ fibers over $\widehat{\Orm}(1)$ such that it decomposes
as
$$C^\infty \left(G'/\tilde{K}', \bigwedge\nolimits^p(\C^n)\right)\cong C^\infty \left(G'/K', \bigwedge\nolimits^p(\C^n)\right) \oplus \left(\chi_{-,+}\otimes C^\infty \left(G'/K', \bigwedge\nolimits^p(\C^n)\right)\right).
$$
Concretely this map is given for $f\in \pi_{p,\lambda}^\pm$ by
$$\Phi f=\Phi_+f+\Phi_-f$$
with 
$$\Phi_+f(g)=\frac{1}{2}( f(g\overline{n}_{e_n}) \pm f(g\tilde{w}_0\overline{n}_{e_n} ))$$
and 
$$\Phi_-f(g)=\frac{1}{2}( f(g\overline{n}_{e_n}) \mp f(g\tilde{w}_0\overline{n}_{e_n} )).$$

Then by restriction to $G'_0$ we obtain the following.
	
\begin{corollary}
As $G'_0$-representations there is a $G'_0$-equivariant linear continuous map
$$\bar{\pi}_{p,\lambda}^\pm|_{G'_0}\to C^\infty \left(G'_0/K'_0, \bigwedge\nolimits^p(\C^n)\right) \oplus C^\infty \left(G'_0/K'_0, \bigwedge\nolimits^p(\C^n)\right),$$
which extends for $\Re(\lambda)>-\frac{1}{2}$ to
$$\bar{\pi}_{p,\lambda}^\pm|_{G'_0}\to L^2\left(G'_0/K'_0, \bigwedge\nolimits^p(\C^n)\right) \oplus L^2 \left(G'_0/K'_0, \bigwedge\nolimits^p(\C^n)\right),$$
which is a unitary map for $\lambda\in i\R$.
\end{corollary}

Let $\pr_{\Orm(1)},\pr_{\Orm(n)}$ denote the projections of $K'\cong M\cong \Orm(1)\times \Orm(n)$ to the $\Orm(1)$ and $\Orm(n)$ factors. Moreover denote $m(g)$ the $M$-factor in the $\Nbar MAN$ decomposition.
\begin{corollary}\label{C:cor:kernel_plancherel} 
	\begin{enumerate}[label=(\roman{*}), ref=\thetheorem(\roman{*})]
		\item We have $$H(e^{-rH}\overline{n}_{(-X,0)})=rH+\log(\abs{(X',e^{-r})}^2)H
		.$$

		\item We have $$\pr_{\Orm(n)} (\kappa(g^{-1}))=w^{-1}\pr_{\Orm(n)}(m(\tilde{w}_0g\overline{n}_{e_n}))w.$$
		
				\item We have for $X_n\in \R^\times$ and $g \in G'$ with $\overline{n}_{(X',X_n)} \in g\overline{n}_{e_n}P$, 
		$$\pr_{\Orm(1)}(\kappa(g^{-1}))=\sgn X_n.
		$$
	\end{enumerate}
\end{corollary}

\begin{proof}
Ad(i): This follows immediately from Lemma~\ref{C:lemma:KAN_decomp}.\\
Ad (ii): By \eqref{eq:K_M_w_conjugation} we have
$$\diag(1,k,1)\overline{n}_{e_n}=\overline{n}_{e_n}\diag(1,1,wkw^{-1}),$$
where $$w=
	\begin{pmatrix}
	 & -\mathbf{1}_{n-1} \\
	 -1 &
	\end{pmatrix}\in \Orm(n),$$
	which implies by the $\Nbar MAN$ decomposition that for all $g\in G$,
	$$\pr_{\Orm(n)}(\overline{\kappa}(g))=w^{-1}\pr_{\Orm(n)}(m(g\overline{n}_{e_n}))w.$$
	Moreover $\kappa(g)=\overline{\kappa}((\tilde{w}_0g)^{-1})\tilde{w_0}$ and since $\tilde{w}_0\in K$ with $\pr_{\Orm(n)}(\tilde{w}_0)=\mathbf{1}_n$ this implies the statement.\\
	Ad (iii):
	We have $$\diag\left(\frac{X_n}{\abs{X_n}},\frac{X_n}{\abs{X_n}},\mathbf{1}_n\right)  \kappa(g^{-1})=\kappa\left(    e^{\log \abs{X_n}H}\overline{n}_{(-X',0)}     \right)=\kappa\left( \overline{n}_{(-\abs{X_n}^{-1}X',0)}     \right).$$
	Then the statement follows from Lemma~\ref{C:lemma:KAN_decomp}.
\end{proof}

\section{A Plancherel formula for \texorpdfstring{$L^2(G'_0/K'_0,{\bar{\sigma}_p})$}{$L^2(G'0/K'0,{s})$}}\label{C:sec:plancherel_G'_0}
We introduce the notation $\bar{\sigma}=\sigma|_{\SO(n)}$ for admissible representations of $\Orm(n)$ and similarly for representations of $\Orm(n-1)$ restricted to $\SO(n-1)$.
In \cite{camporesi_1997} a Plancherel formula for vector bundles over Riemannian symmetric spaces is established and the example of $L^2(\SO_0(1,n)/\SO(n), \bigwedge\nolimits^p(\C^n))$ carried out in great detail. We recall this example in this section.
Let $\phi:{G'_0}\to\End({\bar{\sigma}_p})$ be a spherical function, i.e. satisfying 
\begin{equation}\label{C:eq:def_spherical_function}
\int_{K'_0}\phi(gkh)\,dk=\phi_(g)\phi(h), \qquad \phi(kgk')=\sigma_p(k)\phi(g)\sigma_p(k').
\end{equation}
and normalized to $\phi(\mathbf{1}_{n+1}))=\mathbf{1}$.
\subsection{The Plancherel measure}
Recall that as $\SO(n)$ resp. $\SO(n-1)$-representations we have the isomorphism
$$\bar{\sigma}_p\cong \bar{\sigma}_{n-p}, \qquad \bar{\delta}_q\cong \bar{\delta}_{n-1-q}$$
and that for $n$ even, $\bar{\sigma}_{\frac{n}{2}}$ is reducible and decomposes into two non-isomorphic irreducibles as
$$\bar{\sigma}_\frac{n}{2}=\bar{\sigma}_{\frac{n}{2}}^{(+)}\oplus \bar{\sigma}_{\frac{n}{2}}^{(-)},$$
as well as for $n$ odd, $\bar{\delta}_{\frac{n-1}{2}}$ is reducible and decomposes into two non-isomorphic irreducibles as
$$\bar{\delta}_{\frac{n-1}{2}}=\bar{\delta}_{\frac{n-1}{2}}^{(+)}\oplus \bar{\delta}_{\frac{n-1}{2}}^{(-)}.$$

\begin{lemma}\label{C:lemma:branching_sigma_p_M'_0}
\begin{enumerate}[label=(\roman{*})]
\item For $p\neq \frac{n}{2}, \frac{n\pm1}{2}$ we have
$$\bar{\sigma}_p|_{\SO(n-1)}=\bar{\delta}_{p-1} \oplus \bar{\delta}_{p}.$$
\item For $p=\frac{n}{2}$ we have
$$\bar{\sigma}_\frac{n}{2}^{(\pm)}|_{\SO(n-1)}=\bar{\delta}_{\frac{n}{2}}.$$
\item For $p=\frac{n-1}{2}$ we have
$$\bar{\sigma}_\frac{n-1}{2}|_{\SO(n-1)}=\bar{\delta}_{\frac{n-3}{2}} \oplus \bar{\delta}_{\frac{n-1}{2}}^{(+)}\oplus \bar{\delta}_{\frac{n-1}{2}}^{(+)}.$$
\end{enumerate}
\end{lemma}
Then in the case $p=\frac{n}{2}$ also the bundle
$L^2(G'_0/K'_0,\bar{\sigma}_{\frac{n}{2}})$ is reducible
$$L^2(G'_0/K'_0,\bar{\sigma}_{\frac{n}{2}})\cong L^2(G'_0/K'_0,\bar{\sigma}_{\frac{n}{2}}^{(+)})\oplus L^2(G'_0/K'_0,\bar{\sigma}_{\frac{n}{2}}^{(-)}).$$
The following Plancherel formula holds
\begin{equation}\label{C:eq:plancherel_G/K}
L^2(G'_0/K'_0,\bar{\sigma_p})\cong \int_{\hat{G}'_0(\bar{\sigma}_p)}^\oplus m_{\bar{\sigma}_p}(\tau)\tau\, d\mu_{\bar{\sigma}_p}(\tau), 
\end{equation}
with a Plancherel measure $d\mu_{\bar{\sigma}_p}$, $\hat{G}'_0(\bar{\sigma}_p) \subseteq \hat{G}'_0$ being the support of the measure and $m_{\bar{\sigma}_p}$ the multiplicities.
We denote the corresponding Plancherel measures for $p=\frac{n}{2}$ as
$$\mu_{\bar{\sigma}_{\frac{n}{2}}}=\mu_{\bar{\sigma}_{\frac{n}{2}}^{(+)}}+\mu_{\bar{\sigma}_{\frac{n}{2}}^{(-)}}.
$$

We recall the support and normalization of the Plancerel measure $d\mu_{\sigma_p}$ from \cite[Section 4]{camporesi_1997}. Let $P'_0$ be a minimal parabolic of $G'_0$, for example $P'\cap G'_0$. Consistent with the notation of Section~\ref{C:sec:restriction_identity_comp} $$\bar{\tau}_{q,\nu}=\Ind_{P'_0}^{G'_0}({\bar{\delta}_q} \otimes e^\nu\otimes \mathbf{1})$$
is the principal series representation and we denote for $n$ odd
$$\bar{\tau}_{\frac{n-1}{2},\nu}^{(\pm)}=\Ind_{P'_0}^{G'_0}({\bar{\delta}_{\frac{n-1}{2}}^{(\pm)}} \otimes e^\nu\otimes \mathbf{1}).$$

\begin{prop}\label{C:prop:supp_plancherel_measure}
\begin{enumerate}[label=(\roman{*})]
\item The continuous part of the support of $d\mu_{\bar{\sigma}_p}$ is for $p\neq \frac{n}{2}, \frac{n\pm1}{2}$ given by all $\bar{\tau}_{q,\nu}$ with $q\in \{ p-1,p \} \cap \Z_{\geq 0}$ and $\nu\in i\R$ and all multiplicities are one.
\item The continuous part of the support of $d\mu_{\bar{\sigma}_\frac{n}{2}^{(\pm)}}$ is given by all $\bar{\tau}_{\frac{n}{2},\nu}$ with $\nu \in i\R$ and all multiplicities are one in each case respectively.
\item The continuous part of the support of $d\mu_{\bar{\sigma}_\frac{n-1}{2}}$ is given by all $\bar{\tau}_{\frac{n-1}{2},\nu}^{(\pm)}$ and all $\bar{\tau}_{\frac{n-3}{2},\nu}$ with $\nu \in i\R$ and all multiplicities are one.
\item The discrete part of the support of $d\mu_{\bar{\sigma}_p}$ is empty if and only if $p\neq \frac{n}{2}$. If $p=\frac{n}{2}$, the discrete part of the support of $d\mu_{\bar{\sigma}_p^{(\pm)}}$ is given by $\overline{\Pi}'^{(\pm)}_{\frac{n}{2},+}$. The discrete series representation occurs with multiplicity one in each case respectively.
\end{enumerate}
\end{prop}

Proposition~\ref{C:prop:supp_plancherel_measure} gives an explicit description of the Plancherel formula \eqref{C:eq:plancherel_G/K}. For our purposes we are further interested in the explicit inversion formula.
We therefore  define the $\End(\bar{\sigma}_p)$-valued function $\bar{\phi}_{p,\nu}$ given by
$$\bar{\phi}_{p,\nu}(g)=\int_{K'_0}{\bar{\sigma}_p}(\kappa(gk)k^{-1})e^{(\nu-\rho')H(gk)}\, dk,
$$
which is a spherical function (see e.g. \cite[(3.7)]{shimeno_oda_2018}).
\begin{lemma}\label{C:lemma:spherical_function_conv_formula}
$$\bar{\phi}_{p,\nu}(g^{-1}h)=\int_{K'_0}\bar{\sigma}_p(\kappa(h^{-1}k))e^{(\nu-\rho')H(h^{-1}k)} \bar{\sigma}_p(\kappa(g^{-1}k)^{-1})e^{-(\nu+\rho')H(g^{-1}k)}\,dk.$$
\end{lemma}
\begin{proof}
First note that
$g^{-1}hk=g^{-1}\kappa(hk)e^{H(hk)}n$, and since $A$ normalizes $N'$ we have
\begin{equation}
\kappa(g^{-1}hk)=\kappa(g^{-1}\kappa(hk)),\qquad H(g^{-1}hk)=H(hk)+H(g^{-1}\kappa(hk)), \label{C:eq:KAN_triple_fomrula}
\end{equation}
such that
\begin{multline*}
	\bar{\phi}_{p,\nu}(g^{-1}h)\\=\int_{K'_0} \bar{\sigma}_p(\kappa(h^{-1}\kappa(gk))\bar{\sigma}_p(\kappa(g^{-1}\kappa(gk))^{-1})e^{(-\nu+\rho')H(g^{-1}\kappa(gk))}e^{(\nu-\rho')H(h^{-1}\kappa(gk))}\,dk.
\end{multline*}
By the formula
$$\int_{K'_0}F(\kappa(gk))\,dk=\int_{K'_0} F(k)e^{-2\rho H(g^{-1}k)}\, dk$$
we obtain the Lemma.
\end{proof}
According to Lemma~\ref{C:lemma:spherical_function_conv_formula} we have for $f\in C^\infty_0(G'_0/K'_0,\bar{\sigma}_p)$,
$$\bar{\phi}_{p,\nu}\ast f(h)=\int_{K'_0}\bar{\sigma}_p(\kappa(h^{-1}k))e^{(\nu-\rho')H(h^{-1}k)} \int_{G'_0} \bar{\sigma}_p(\kappa(g^{-1}k)^{-1})e^{-(\nu+\rho')H(g^{-1}k)}f(g)\,dg\,dk$$
and we define the corresponding Fourier transform by
$$\tilde{f}(k,p,\nu):=\int_{G'_0}\bar{\sigma}_p(\kappa(g^{-1}k)^{-1})e^{-(\nu+\rho')H(g^{-1}k)}f(g)\,dg.$$
Then clearly for $f\in C^\infty_0(G'_0/K'_0,\bar{\sigma}_p)$ and $man\in M'AN'$,
$$\tilde{f}(kman,p,\nu)=\bar{\sigma}_p(m^{-1})a^{-(\nu+\rho')}\tilde{f}(k,p,\nu),$$
such that the Fourier transform defines a $G'_0$-intertwining operator
$$C^\infty_0(G'_0/K'_0,\bar{\sigma}_p)\to \Ind_{P'_0}^{G'_0}(\bar{\sigma}_p|_{M'_0}\otimes e^\nu\otimes \mathbf{1}).$$
Now $\bar{\sigma}_p|_{M'_0}$ is reducible and decomposes into $\SO(n-1)$-representations according to Lemma~\ref{C:lemma:branching_sigma_p_M'_0}. And on the one hand if $\bar{\delta}$ is a $M'_0$-representation occuring in $\bar{\sigma}_p|_{M'_0}$ and $\Ind_{P'_0}^{G'_0}(\bar{\delta} \otimes e^\nu\otimes \mathbf{1})\in \hat{G}'_0(\bar{\sigma}_p)$, every other principal series $\Ind_{P'_0}^{G'_0}(\bar{\delta}' \otimes e^\nu\otimes \mathbf{1})\in \hat{G}'_0(\bar{\sigma}_p)$ for all other $\bar{\delta}'$ occurring in $\bar{\sigma}_p|_{M'_0}$ according to Proposition~\ref{C:prop:supp_plancherel_measure}.
Applying these results to the Plancherel formula \eqref{C:eq:plancherel_G/K} and the corresponding inversion formula \cite[(39)]{camporesi_1997} we obtain the following Theorem.

\begin{theorem}[Inversion formula]\label{C:theorem:inversion_formula}
We have for $p\neq \frac{n}{2}$ $$L^2(G'_0/K'_0,\bar{\sigma}_p)\simeq \int^\oplus_{i\R_+} L^2-\Ind_{P'_0}^{G'_0}(\bar{\sigma}_p|_{M'_0} \otimes e^\nu\otimes \mathbf{1})\,d\mu_{\bar{\sigma}_p}(\nu)$$
and for all $f\in C_0^\infty(G'_0/K'_0,\bar{\sigma}_p)$
$$f(g)=\int_{i\R} \bar{\phi}_{p,\nu}\ast f(g) \,d\mu_{\bar{\sigma}_p}(\nu).$$

For $p=\frac{n}{2}$ we have
$$L^2(G'_0/K'_0,\bar{\sigma}_p^{(\pm)})\simeq \int^\oplus_{i\R_+} L^2-\Ind_{P'_0}^{G'_0}(\bar{\sigma}_p^{(\pm)}|_{M'_0} \otimes e^\nu\otimes \mathbf{1})\,d\mu_{\bar{\sigma}_p^{(\pm)}}(\nu)\oplus \widehat{\overline{\Pi}}'^{(\pm)}_{\frac{n}{2},+} $$
and for all $f\in C_0^\infty(G'_0/K'_0,\bar{\sigma}_p^{(\pm)})$
$$f(g)=\int_{i\R}  \bar{\phi}_{p,\nu}\ast f(g) \,d\mu_{\bar{\sigma}_p^{(\pm)}}(\nu)+c_p  \bar{\phi}_{p,\frac{1}{2}}\ast f(g)$$
with $c_p\in \C$ a constant.
\end{theorem}

Following \cite[Example 4.4]{camporesi_1997} we have explicitly for $p\neq
\frac{n}{2}$, $p\in \{  1,\dots, \lfloor\frac{n-1}{2}\rfloor \}$,
$$d\mu_{\bar{\sigma}_p}(\nu)={{n-1}\choose{p}}\frac{d\nu}{c(p,\nu)c(p,-\nu)},$$
with $c$-function
$$c(p,\nu)=2^{-n+2}\frac{ \Gamma(\frac{n}{2})(\nu+\rho'-p)\Gamma(\nu)   }{\Gamma(\nu+\rho'+1)}$$
and for $p=\frac{n}{2}$
$$d\mu_{\bar{\sigma}_p^{(\pm)}}(\nu)=\frac{1}{2}{{n-1}\choose{p}}\frac{d\nu}{c(p,\nu)c(p,-\nu)},$$
with $c$-function as above and discrete constant
$$c_\frac{n}{2}=2^{-n}\frac{n!}{(\frac{n}{2})!}\prod_{s=1}^{\frac{n}{2}-1}(2s)!.$$

\section{The Plancherel formula for \texorpdfstring{$L^2(G'/K',{{\sigma}_p})$}{$L^2(G'/K',{s})$}}\label{C:sec:plancherel_G'}
In this section we lift the results of the previous section to the disconnected group $G'$.
we choose representatives $\tilde{v}_0,\tilde{w}_0\in K'$ generating the component group $G'/G'_0$ given by $$\tilde{w}_0=\diag(-1,\mathbf{1}_{n+1}), \qquad \tilde{v}_0=\diag(-1,\tilde{m}),$$
with $\tilde{m}=\diag(-1,\mathbf{1}_{n})$.
For $f \in L^2(G'_0/K'_0,\bar{\sigma}_p)$ we define for $g\in G'_0$
$$f(\tilde{w}_0g):=f(\tilde{w}_0g\tilde{w}_0^{-1})$$
and 
$$f(\tilde{v}_0g):=\sigma(\tilde{m}^{-1})f(\tilde{v}_0g\tilde{v}_0^{-1}),$$
where $\sigma \in \{  \sigma_p,\sigma_{n-p}  \}$, such that $\sigma|_{\SO(n)}=\bar{\sigma}_p.$

Moreover we define the $\End(\sigma_p)$-valued function on $G'$
$${\phi}_{p,\nu}(g)=\int_{K'}{{\sigma}_p}(\kappa(gk)k^{-1})e^{(\nu-\rho')H(gk)}\, dk.
$$

\begin{theorem}\label{C:theorem:plancherel_O(1,n)}
We have for $p\neq \frac{n}{2}$ $$L^2(G'/K',{\sigma}_p)\simeq \int^\oplus_{i\R_+} L^2-\Ind_{P'}^{G'}(\sigma_p|_{M'} \otimes e^\nu\otimes \mathbf{1})\,d\mu_{\sigma_p}(\nu)$$
and for all $f\in C_0^\infty(G'/K',{\sigma}_p)$
$$f(g)=\int_{i\R} {\phi}_{p,\nu}\ast f(g) \,d\mu_{{\sigma}_p}(\nu).$$

For $p=\frac{n}{2}$ we have
$$L^2(G'/K',{\sigma}_p)\simeq\int^\oplus_{i\R_+} L^2-\Ind_{P'}^{G'}({\sigma}_p|_{M'} \otimes e^\nu
\otimes \mathbf{1})\,d\mu_{{\sigma}_p}(\nu)\oplus   \Pi'_{\frac{n}{2},+}$$
and for all $f\in C_0^\infty(G'/K',{\sigma}_p)$
$$f(g)=\int_{i\R}  {\phi}_{p,\nu}\ast f(g) \,d\mu_{{\sigma}_p}(\nu)+c_p  {\phi}_{p,\frac{1}{2}}\ast f(g)$$
with $c_p\in \C$ as before and $$d\mu_{\sigma_p}(\nu)=d\mu_{\bar{\sigma}_{\min(p,n-p)}}(\nu)$$ in the notation of the last section.
\end{theorem}

\begin{proof}
Let $p\neq \frac{n}{2}$ and w.l.o.g. $p < n-p$.
Let $f\in C_0^\infty(G'/K',\sigma_p)$ and $h=h_ch_0 \in G'$ with $h_0\in G'_0$ and $h_c\in G'/G'_0$. Then by the construction above
\begin{align*}
f(h)&=\sigma_p(h_c^{-1})\int_{i\R} \bar{\phi}_{p,\nu}\ast f(h_ch_0h_c^{-1}) \,d\mu_{\bar{\sigma}_p}(\nu) \\
&=\sigma_p(h_c^{-1})\int_{i\R}\int_{K'_0}\bar{\sigma}_p(\kappa(h_ch_0^{-1}h_c^{-1}k))e^{(\nu-\rho')H(h_ch_0^{-1}h_c^{-1}k)}\\
&\hspace{4.4cm}\times \int_{G'_0} \bar{\sigma}_p(\kappa(g^{-1}k)^{-1})e^{-(\nu+\rho')H(g^{-1}k)}f(g)\,dg\,dk\,d\mu_{\bar{\sigma}_p}(\nu).\\
\intertext{The $K'_0$-integral is right $M'_0$-invariant and the $G'_0$-integral is right $K'_0$-invariant. Moreover $K'_0/M'_0=K'/M'$ and $G'_0/K'_0=G'/K'$ and since $\bar{\sigma}_p=\sigma_p|_{\SO(n)}$, if we replace $\bar{\sigma}_p$ by ${\sigma}_p$ we obtain right $M'$ and right $K'$ invariant integrals }
&=\int_{i\R}\int_{K'}{\sigma}_p(\kappa(h^{-1}k))e^{(\nu-\rho')H(h^{-1}k)} \\ &\hspace{4.4cm}\times \int_{G'} {\sigma}_p(\kappa(g^{-1}k)^{-1})e^{-(\nu+\rho')H(g^{-1}k)}f(g)\,dg\,dk\,d\mu_{{\sigma}_p}(\nu)
\\ &= \int_{i\R} \phi_{p,\nu}\ast f(h)\,d\mu_{\bar{\sigma}_p}(\nu).
\end{align*}
For $p=\frac{n}{2}$ the proof works in the same way using the direct sum $\bar{\sigma}_{\frac{n}{2}}=\bar{\sigma}_{\frac{n}{2}}^{(+)}\oplus  \bar{\sigma}_{\frac{n}{2}}^{(-)}    $ and  carrying the discrete summand through the calculation. If $p>n-p$ we have to apply the $\SO_0(1,n)$-Plancherel and inversion formula for $\bar{\sigma}_{n-p}$ which concludes the argument.
\end{proof}

Similarly we define the corresponding Fourier-transform for $f\in C^\infty(G'/K',\sigma_p^w)$ by
$$\tilde{f}(k,p,\nu)=\int_{G'} \sigma_p^w(\kappa(g^{-1}k)^{-1})e^{(-\nu-\rho')H(g^{-1}k)}f(g)\,dg.$$
Then clearly for $\nu\in i\R$
\begin{equation}\label{C:eq:norm_fourier}
\langle \phi_{p,\nu}\ast f,f\rangle_{L^2(G')}=\norm{  \tilde{f}(\cdot,p,\nu)   }^2_{L^2(K')}.
\end{equation}
We remark that we use the equivalent representation $\sigma_p^w$ which is twisted by $w$ for convenience in the following.

\section{Branching laws for unitary representations}\label{C:sec:branching_laws}
We lift the results of the section before to $\pi_{p,\lambda}^{\pm}$ and prove the main theorems.

\begin{theorem}\label{C:theorem:coordinate_change}
Let $\Re(\lambda)>-\frac{1}{2}$ and $f\in \pi_{p,\lambda}^\pm$. We have
$$2\tilde{\Phi_\pm f}(\cdot, p, \nu)=A_{(p,\lambda),(p-1,\nu)}^\mp f+A_{(p,\lambda),(p,\nu)}^\pm f.$$
\end{theorem}

\begin{proof}
We carry out the proof for $\alpha=\mathbf{1}$ since the other case works analogously.
Let $f \in \pi_{p,\lambda}^+$. By Lemma~\ref{C:lemma:L^2_condition} $\Phi_+f\in L^2(G'/K',\sigma_p^w)$ such that we can apply the Fourier transform.
Clearly the integrand is right $K'$-invariant in $g$ such that we have by the $\Nbar'AK'$ Iwasawa decomposition and by the integral formula
$$\int_{G'/K'}f(g)\,dg=\int_{\Nbar\times \mathfrak{a}} e^{2\rho'X}f(\overline{n} e^{X} )\,d\overline{n}\,dX,$$
\begin{align*}
& \int_{G'}    {\sigma_p}^w(\pr_{\Orm(n)}(\kappa(g^{-1}h))  e^{-(\nu+\rho')H(g^{-1}h)} \Phi_+f(g)\,dg \\
=&\int_{\R^{n-1}\times \R}
{\sigma_p}^w(\pr_{\Orm(n)}(\kappa(e^{-rH}\overline{n}_{(-X',0)}))   e^{-(\nu+\rho')H(e^{-rH}\overline{n}_{(-X',0)})} \Phi_+f(h\overline{n}_{(X',0)}e^{rH})\,dX'\,dr \\
 \intertext{which is by Corollary~\ref{C:cor:kernel_plancherel}}
 =&\int_{\R^{n-1}\times \R} {\sigma_p}(\pr_{\Orm(n)}(m(\tilde{w}_0\overline{n}_{(X',0)}e^{rH}\overline{n}_{e_n})) \\ & \hspace{4.5cm}\times e^{(-\nu+\rho')r}\abs{e^{-2r}+\abs{X'}^2}^{{-\nu-\rho'}}  \Phi_+f(h\overline{n}_{(X',0)}e^{rH})\,dX'\,dr \\
 \intertext{which is by Lemma~\ref{C:prop:matrix_decompositions}}
=&\frac{1}{2}\int_{\R^{n-1}\times \R_+} {\sigma_p}(\psi_n(X',X_n)) \abs{X_n}^{\lambda-\rho-\nu-\rho'}\abs{\abs{X_n}^2+\abs{X'}^2}^{{-\nu-\rho'}}
\\
& \hspace{4cm} \times (f(h\overline{n}_{(X',\abs{X_n})})+ f(h\overline{n}_{(X',-\abs{X_n})} \diag (\mathbf{1}_{n+1},-1))\,dX'\,dX_n
\\
=&\frac{1}{2}\int_{\R^{n-1}\times \R_+} {\sigma_p}(\psi_n(X',X_n)) \abs{X_n}^{\lambda-\rho-\nu-\rho'}\abs{\abs{X_n}^2+\abs{X'}^2}^{{-\nu-\rho'}}
\\
& \hspace{2.6cm} \times (f(h\overline{n}_{(X',\abs{X_n})})+ (-1)^p \sigma_p(\diag(-\mathbf{1}_{n-1},1))f(h\overline{n}_{(X',-\abs{X_n})})\,dX'\,dX_n
\intertext{
Since $\sigma_P|_{M'}=\delta_{p-1}\oplus \delta_p$ we project to the two subspaces separately. Since $\delta_{q}(-\mathbf{1}_{n-1})=(-1)^{q}$ we obtain
}
\\
=&\frac{1}{2}\int_{\R^{n-1}\times \R_+} \pr_{\sigma_p\to\delta_p} \circ {\sigma_p}(\psi_n(X',X_n)) \abs{X_n}^{\lambda-\rho-\nu-\rho'}\abs{\abs{X_n}^2+\abs{X'}^2}^{{-\nu-\rho'}}
\\
& \hspace{6.5cm} \times   (f(h\overline{n}_{(X',\abs{X_n})})+ f(h\overline{n}_{(X',-\abs{X_n})})\,dX'\,dX_n
\\&+ \frac{1}{2}\int_{\R^{n-1}\times \R_+} \pr_{\sigma_p\to\delta_{p-1}} \circ {\sigma_p}(\psi_n(X',X_n)) \abs{X_n}^{\lambda-\rho-\nu-\rho'}\abs{\abs{X_n}^2+\abs{X'}^2}^{{-\nu-\rho'}}
\\
& \hspace{6.5cm} \times  (f(h\overline{n}_{(X',\abs{X_n})})- f(h\overline{n}_{(X',-\abs{X_n})})\,dX'\,dX_n
\\=&\frac{1}{2}\int_{\R^{n}} \pr_{\sigma_p\to\delta_p} \circ{\sigma_p}(\psi_n(X)) \abs{X_n}^{\lambda-\rho-\nu-\rho'}\abs{X}^{-2{(\nu+\rho')}}
f(h\overline{n}_X)\,dX
\\&\hspace{2cm}+\frac{1}{2}\int_{\R^{n}} \pr_{\sigma_p\to\delta_{p-1}} \circ{\sigma_p}(\psi_n(X)) \abs{X_n}^{\lambda-\rho-\nu-\rho'}\abs{X}^{-2{(\nu+\rho')}} \sgn(X_n)
f(h\overline{n}_X)\,dX. \\
\intertext{which is by Proposition~\ref{prop:A_action} equal to}
=&\frac{1}{2}\left(A_{(p,\lambda),(p,\nu)}^+f\right)(h)+\frac{1}{2}\left(A_{(p,\lambda),(p-1,\nu)}^-f\right)(h) \qedhere
\end{align*}
\end{proof}

Combining this result with Lemma~\ref{C:lemma:L^2_condition} and Theorem~\ref{C:theorem:plancherel_O(1,n)} we immediately obtain the unitary branching law and Plancherel formula for the unitary principal series.
Therefore we define the following functions which depend meromorphically on $\lambda$ and $\nu$.
\begin{multline*}
	c(p,\lambda,\nu)^\pm:=\\\frac{c(\min\{p,n-p\},-\nu)c(\min\{p,n-p\},\nu)}{\Gamma(\frac{-\lambda+\rho-\nu-\rho'}{2}-\frac{\pm 1 -1}{4})
		\Gamma(\frac{-\lambda+\rho+\nu-\rho'}{2}-\frac{\pm 1 -1}{4})
		\Gamma(\frac{\lambda+\rho-\nu-\rho'}{2}-\frac{\pm 1 -1}{4})
		\Gamma(\frac{\lambda+\rho+\nu-\rho'}{2}-\frac{\pm 1 -1}{4})
	},
\end{multline*}
$$c(p,\lambda,q,k)_{Res}^\pm:=\pi
\Res_{\mu=\lambda+1-(\pm\frac{1}{2})+2k}\left(\frac{1}{c(p,\lambda,\mu)^\pm t'(p,q,\mu)c_C(p,q,\mu)^2}\right)
,$$
$$c(\lambda)_{d}:=\frac{c_{\frac{n}{2}} \Gamma(\rho)\Gamma(\frac{-\lambda+1}{2})  \Gamma(\frac{-\lambda+2}{2}) \Gamma(\frac{\lambda+1}{2}) \Gamma(\frac{\lambda+2}{2})      }
{ 2 \pi^{\frac{n-1}{2}}
}.$$

The following can be easily read off the definitions of these scalars.
\begin{lemma}\label{C:lemma:scalars_positive}
	\begin{enumerate}[label=(\roman{*})]
		\item Let $p=0$ and $k\in \Z_{\geq 0}$. The function $c(0,\lambda,0,k)^+_{Res}$ is holomorohic in $\lambda$ in the range $\Re\lambda<-\frac{1}{2}-2k$ and strictly positive for $\Im \lambda=0$. The function $c(0,\lambda,0,k)^-_{Res}$ is holomorphic in $\lambda$ in the range $\Re\lambda <-\frac{3}{2}-2k$ and strictly positive for $\Im \lambda=0$.
		\item Let $0<p<\frac{n-1}{2}$, $q=p-1,p$ and $k\in \Z_{\geq 0}$. The function $c(p,\lambda,q,k)^+_{Res}$  holomorphic in $\lambda$ in the range $\Re\lambda\in [p-\rho,-\frac{1}{2}-2k)$ and  strictly positive for $\Im \lambda=0$. The function $c(p,\lambda,q,k)^-_{Res}$ is holomorphic in $\lambda$ in the range $\Re\lambda\in [p-\rho,-\frac{3}{2}-2k)$and  strictly positive for $\Im \lambda=0$.
		\item Let $\frac{n+1}{2}<p<n$, $q=p-1,p$ and $k\in \Z_{\geq 0}$. The function $c(p,\lambda,q,k)^+_{Res}$ is holomorphic in $\lambda$ in the range $\Re\lambda\in [\rho-p,-\frac{1}{2}-2k)$ and strictly negative for $\Im \lambda=0$. The function $c(p,\lambda,q,k)^-_{Res}$ is holomorphic in $\lambda$ in the range $\Re\lambda\in [\rho-p,-\frac{3}{2}-2k)$ and strictly negative for $\Im \lambda=0$.
		\item Let $p=n$ and $k\in \Z_{\geq 0}$. The function $c(n,\lambda,n-1,k)^+_{Res}$ is holomorphic in $\lambda$ in the range $\Re\lambda <-\frac{1}{2}-2k$ and strictly negative for $\Im \lambda=0$.The function $c(n,\lambda,n-1,k)^-_{Res}$ is holomorphic in $\lambda$ in the range $\Re\lambda <-\frac{3}{2}-2k$ and strictly negative for $\Im \lambda=0$.
	\end{enumerate}
\end{lemma}
\begin{remark}
	For example we have $$c(0,\lambda,0,k)^+_{Res}=\frac{2^{2n-3}(-\lambda-\frac{1}{2}-2k)\Gamma(-\lambda+\rho-1-2k)\Gamma(-\lambda-k)\Gamma(k+\frac{1}{2})}{\pi^{\frac{n-3}{2}}\Gamma(\frac{n}{2})^2k!\Gamma(-\lambda-k+\frac{1}{2})}$$
	and $$c(0,\lambda,0,k)^-_{Res}=\frac{2^{2n-3}(-\lambda-\frac{1}{2}-2k)\Gamma(-\lambda+\rho-1-2k)\Gamma(-\lambda-k+1)\Gamma(k+\frac{3}{2})}{\pi^{\frac{n-3}{2}}\Gamma(\frac{n}{2})^2k!\Gamma(-\lambda-k+\frac{1}{2})}.$$
	We don't give the explicit expressions in every case for the sake of the length and readability of this article.
\end{remark}
\begin{lemma}\label{C:lemma:branching_law_unitary_ps}
For $\lambda \in i\R $  and $p\neq \frac{n}{2}$ we have
$$\hat{\pi}_{p,\lambda}^\pm|_{G'}\simeq \bigoplus_{\alpha=\pm}\bigoplus_{q=p-1,p}\int^\oplus_{i\R_+} \hat{\tau}_{q,\nu}^\alpha \, d\mu_{\sigma_p}(\nu)$$

and for $f\in \pi_{p,\lambda}^\pm$
$$
\norm{f}_{L^2(K)}^2=\frac{1}{4}\sum_{\alpha=+,-}\sum_{q=p-1,p}\int_{i\R} \norm{\tilde{A}_{(p,\lambda),(q,\nu)}^\alpha f}_{L^2(K')}^2 \frac{d\nu}{c(p,\lambda,\nu)^\alpha}.
$$

For $\lambda \in i\R  $  and $p=\frac{n}{2}$ we have
$$\hat{\pi}_{p,\lambda}^\pm|_{G'}\simeq   \widehat{\Pi}'_{\frac{n}{2},+} \oplus \widehat{\Pi}'_{\frac{n}{2},-} \oplus \bigoplus_{\alpha=\pm}\bigoplus_{q=p-1,p}\int^\oplus_{i\R_+} \hat{\tau}_{q,\nu}^\alpha \, d\mu_{\sigma_p}(\nu) $$

and for $f\in \pi_{p,\lambda}^\pm$
\begin{multline*}
\norm{f}_{L^2(K)}^2=\frac{1}{4}\sum_{\alpha=+,-}\sum_{q=p-1,p} \int_{i\R} \norm{\tilde{A}_{(p,\lambda),(q,\nu)}^\alpha f}_{q,\nu}^2 \frac{d\nu}{c(p,\lambda,\nu)^\alpha} \\+c(\lambda)_d\norm{\tilde{A}_{(p,\lambda),(p,-\frac{1}{2})}^{-,quo} f}_{p,-\frac{1}{2},quo}^2+c(\lambda)_d\norm{\tilde{A}_{(p,\lambda),(p-1,-\frac{1}{2})}^{-,quo} f}_{p-1,-\frac{1}{2},quo}^2.
\end{multline*}
\end{lemma}

\begin{proof}
Let $\lambda \in i\R$ and $f\in \pi_{p,\lambda}^\alpha$. Then by Lemma~\ref{C:lemma:L^2_condition} ${\Phi}={\Phi}_++{\Phi}_-$ is a unitary map such that by orthogonality
$$\norm{f}_{L^2(K)}^2=\norm{\Phi f}_{L^2(G')}^2=\norm{\Phi_+ f}_{L^2(G')}^2+\norm{\Phi_- f}_{L^2(G')}^2.$$
Let $p\neq \frac{n}{2}$. Then applying the inversion formula of Theorem~\ref{C:theorem:plancherel_O(1,n)} and Theorem~\ref{C:theorem:coordinate_change} and \eqref{C:eq:norm_fourier} we obtain
\begin{multline*}
\norm{\Phi_\pm f}_{L^2(G')}^2=\frac{1}{4} \int_{i\R} \left(\norm{A_{(p,\lambda),(p-1,\nu)}^\mp f}_{L^2(K')}^2 +   \norm{A_{(p,\lambda),(p,\nu)}^\pm f}_{L^2(K')}^2 \right) \, d\mu_{\sigma_p}(\nu) \\
=\frac{1}{4} \left( \int_{i\R} \norm{\tilde{A}_{(p,\lambda),(p-1,\nu)}^\mp f}_{L^2(K')}^2 \frac{d\nu}{c(p,\lambda,\nu)^\mp} +  \int_{i\R} \norm{\tilde{A}_{(p,\lambda),(p,\nu)}^\pm f}_{L^2(K')}^2 \frac{d\nu}{c(p,\lambda,\nu)^\pm} \right) 
\end{multline*}
by renormalization to holomorphic families.
For $p=\frac{n}{2}$ the argument for the continuous part of the Plancherel formula is the same. For the discrete summands we reformulate for $\lambda \neq 0$
\begin{multline*}
\langle \phi_{p,\frac{1}{2}} \ast \Phi_+ f,\Phi_+ f\rangle_{L^2(G')}=\left\langle \tilde{\Phi_+ f}\left(\cdot,p,\frac{1}{2}\right),\tilde{\Phi_+ f}\left(\cdot,p,-\frac{1}{2}\right) \right\rangle_{L^2(K')}
\\
=\langle A_{(p,\lambda),(p-1,\frac{1}{2})}^- f, A_{(p,\lambda),(p-1,-\frac{1}{2})}^- f\rangle_{L^2(K')} +\langle A_{(p,\lambda),(p,\frac{1}{2})}^+ f, A_{(p,\lambda),(p,-\frac{1}{2})}^+ f\rangle_{L^2(K')}.
\end{multline*}
Then 
\begin{multline*}
\langle A_{(p,\lambda),(p,\frac{1}{2})}^+ f, A_{(p,\lambda),(p,-\frac{1}{2})}^+ f\rangle_{L^2(K')}\\=\Gamma\left(\frac{-\lambda}{2}\right)  \Gamma\left(\frac{-\lambda+1}{2}\right)\Gamma\left(\frac{\lambda}{2}\right)\Gamma\left(\frac{\lambda+1}{2}\right)\langle \tilde{A}_{(p,\lambda),(p,\frac{1}{2})}^+ f, \tilde{A}_{(p,\lambda),(p,-\frac{1}{2})}^+ f\rangle_{L^2(K')} \\
\end{multline*}
Since the image of $ \tilde{A}_{(p,\lambda),(p,\frac{1}{2})}^+$ is $\Pi'_{p,\alpha}$, we have
$$\langle \tilde{A}_{(p,\lambda),(p,\frac{1}{2})}^+ f, \tilde{A}_{(p,\lambda),(p,-\frac{1}{2})}^+ f\rangle_{L^2(K')}=\langle \tilde{A}_{(p,\lambda),(p,\frac{1}{2})}^+ f, \tilde{A}_{(p,\lambda),(p,-\frac{1}{2})}^{+,quo} f\rangle_{L^2(K')}.$$
Then applying Proposition~\ref{C:prop:func_equations_quotients} we obtain
\begin{multline*}
\langle \tilde{A}_{(p,\lambda),(p,\frac{1}{2})}^+ f, \tilde{A}_{(p,\lambda),(p,-\frac{1}{2})}^{+,quo} f\rangle_{L^2(K')}= -\frac{\Gamma(\rho)}{\pi^{\frac{n-1}{2}}}
\langle T'^{quo}_{p,-\frac{1}{2}} \circ \tilde{A}_{(p,\lambda),(p,-\frac{1}{2})}^{+,quo} f, \tilde{A}_{(p,\lambda),(p,-\frac{1}{2})}^{+,quo} f\rangle_{L^2(K')}\\
=\frac{\Gamma(\rho)}{\pi^{\frac{n-1}{2}}} \norm{ \tilde{A}_{(p,\lambda),(p,-\frac{1}{2})}^{+,quo} f}^2_{p,-\frac{1}{2},quo}= \frac{4\Gamma(\rho)}{\lambda(-\lambda)\pi^{\frac{n-1}{2}}} \norm{ \tilde{A}_{(p,\lambda),(p-1,-\frac{1}{2})}^{-,quo} f}^2_{p-1,-\frac{1}{2},quo}
\end{multline*}
by Proposition~\ref{C:prop:norms_equal}.
Similarly
\begin{multline*}
\langle A_{(p,\lambda),(p-1,\frac{1}{2})}^- f, A_{(p,\lambda),(p-1,-\frac{1}{2})}^- f\rangle_{L^2(K')}
\\=
\frac{ \Gamma(\rho)\Gamma(\frac{-\lambda+1}{2})  \Gamma(\frac{-\lambda+2}{2}) \Gamma(\frac{\lambda+1}{2}) \Gamma(\frac{\lambda+2}{2})      }
{ \pi^{\frac{n-1}{2}}} \norm{ \tilde{A}_{(p,\lambda),(p-1,-\frac{1}{2})}^{-,quo} f}^2_{p-1,-\frac{1}{2},quo}.
\end{multline*}
For $\Phi_- f$ the argument works analogously.
\end{proof}

\section{Analytic continuation}\label{C:sec:analytic_continuation}
Since we have by Lemma~\ref{C:lemma:L^2_condition} good behavior for $\Re(\lambda)>-\frac{1}{2}$ we can extend this Plancherel formula onto the real axis. We remark that we abuse notation and write $\norm{f}_{p,\lambda}$ for $\langle f ,f \rangle_{p,\lambda}$ even if the bilinear pairing is not a norm.

\begin{corollary}\label{C:cor:plancherel_1/2}
For $\Re\lambda \in  (-\frac{1}{2},\frac{1}{2})$, $p\neq \frac{n}{2}$ 
and $f\in \pi_{p,\lambda}^\pm$ we have
$$
\norm{f}_{p,\lambda}^2=\frac{1}{4}\sum_{\alpha=+,-}\sum_{q=p-1,p}\int_{i\R} \norm{\tilde{A}_{(p,\lambda),(q,\nu)}^\alpha f}_{L^2(K')}^2 \frac{\abs{t(p,q,\lambda)}}{c(p,\lambda,\nu)^\alpha}d\nu.
$$
\end{corollary}

\begin{proof}
In the following we abbreviate the integral pairings $$\int_\ast f(h)g(h)\,dh$$ by $(f,g)_\ast$ for the Lie group $\ast$.
First by the same calculation as in the proof of Lemma~\ref{C:lemma:L^2_condition} we have for $\lambda \in \R$ and $f\in \pi_{p,\lambda}^\pm$ that
$$(\Phi f, \Phi \circ T_{p,\lambda} \overline{f})_{G'}=\norm{f}^2_{p,\lambda},$$ for $p< \frac{n}{2}$ and 
$$(\Phi f, \Phi \circ -T_{p,\lambda} \overline{f})_{G'}=\norm{f}^2_{p,\lambda},$$ for $p> \frac{n}{2}$,
since $T_{p,\lambda} \overline{f}  \in \pi_{p,-\lambda}^\pm$. Moreover by Lemma~\ref{C:lemma:L^2_condition} both $\Phi f\in L^2(G'/K',\sigma_p)$ and $\Phi\circ T_{p,\lambda}f \in L^2(G'/K',\sigma_p)$ for $\Re(\lambda)\in (-\frac{1}{2},\frac{1}{2})$ such that we can apply the inversion formula of Theorem~\ref{C:theorem:plancherel_O(1,n)} to $f$ and exchange orders of integrals. For $p\neq \frac{n}{2}$ we obtain

\begin{equation*}
(f,T_{p,\lambda}\overline{f})_K=\frac{1}{4}\sum_{\alpha=+,-}\sum_{q=p-1,p}\int_{i\R} (\tilde{A}_{(p,\lambda),(q,\nu)}^\alpha f,\tilde{A}_{(p,-\lambda),(q,-\nu)}^\alpha \circ T_{p,\lambda}\overline{f})_{K'} \frac{d\nu}{c(p,\lambda,\nu)^\alpha}
\end{equation*}
in the same way as in Lemma~\ref{C:lemma:branching_law_unitary_ps}.
Applying the functional equation for $T_{p,\lambda}$ of Theorem~\ref{C:theorem:functional_equations} this proves the statement since $t(p,q,\alpha)\geq 0$ for all $p<\frac{n}{2}$ and $t(p,q,\alpha)\leq 0$ for all $p>\frac{n}{2}$.
\end{proof}

Rewriting the parings of the Plancherel formula in the Corollary above as integral pairings we obtain an equality for $\lambda \in (-\frac{1}{2}, \frac{1}{2})$
\begin{equation}\label{C:eq:plancherel_induction_start}
(f,T_{p,\lambda}\overline{f})_K=\frac{1}{4}\sum_{\alpha=+,-}\sum_{q=p-1,p}\int_{i\R} (\tilde{A}_{(p,\lambda),(q,\nu)}^\alpha f,\tilde{A}_{(p,-\lambda),(q,-\nu)}^\alpha \circ T_{p,\lambda}\overline{f})_{K'} \frac{t(p,q,\lambda)}{c(p,\lambda,\nu)^\alpha}d\nu.
\end{equation}
In this sense the left hand side of this equation is holomorphic in $\lambda$ if we consider $f$ as a function in the compact picture, i.e. as a function on $K/M$. The right hand side on the other hand is meromorphic in $\lambda$ and has its meromorphic structure governed by the function $c(p,\lambda,\nu)^\pm$ since $t(p,q,\lambda)$ is a regular function for $\Re(\lambda) \leq \frac{1}{2}$. Hence we can analytically continue the right hand side towards $\lambda \in (-\infty, 0)$, where the left hand side is essentially $\norm{f}_{p,\lambda}$, to obtain Plancherel formulas on the whole complementary series and on unitarizable quotients.

\begin{prop}\label{C:prop:plancherel_general}
For $p\neq 0,\frac{n}{2},n$ with $\lambda \in [-\abs{\rho-p}, \frac{1}{2})$ and for $p=0,n$ with $\lambda\in(-\infty,\frac{1}{2})$
\begin{multline*}(f,T_{p,\lambda}\overline{f})_K=\frac{1}{4}\sum_{\alpha=+,-}\sum_{q=p-1,p}\Bigg(\int_{i\R} \norm{\tilde{A}_{(p,\lambda),(q,\nu)}^\alpha f}_{L^2(K')}^2 \frac{t(p,q,\lambda)}{c(p,\lambda,\nu)^\alpha}d\nu
\\
+\sum_{  k \in [0, \frac{-\lambda-1+(\alpha \frac{1}{2})}{2})\cap \Z       }   t(p,q,\lambda)c(p,\lambda,q,k)_{Res}^\alpha\\ \times \norm{C_{(p,\lambda),(q,\lambda+1-(\alpha\frac{1}{2})+2k)}^\alpha f}_{q,\lambda+1-(\alpha\frac{1}{2})+2k}^2\Bigg).
\end{multline*}
\end{prop}

\begin{proof}
	By Corollary~\ref{C:cor:plancherel_1/2} we have for $\Re \lambda \in (-\frac{1}{2},\frac{1}{2})$
	$$(f,T_{p,\lambda}\overline{f})_K=\frac{1}{4}\sum_{\alpha=+,-}\sum_{q=p-1,p}\int_{i\R} (\tilde{A}_{(p,\lambda),(q,\nu)}^\alpha f,\tilde{A}_{(p,-\lambda),(q,-\nu)}^\alpha \circ T_{p,\lambda}\overline{f})_{K'} \frac{t(p,q,\lambda)}{c(p,\lambda,\nu)^\alpha}d\nu.$$
	We prove that this has a holomorphic extension to $\Re\lambda <0$ if $p=0,n$ and to $\Re\lambda \in [-\abs{\rho-p},0)$ in the other cases. We prove for each of the integrals that this extension is  for $\Re\lambda \notin -1+(\alpha \frac{1}{2})-2\Z_{\geq 0}$ given by
	\begin{multline*}
		\int_{i\R} (\tilde{A}_{(p,\lambda),(q,\nu)}^\alpha f,\tilde{A}_{(p,\lambda),(q,-\nu)}^\alpha \overline{f})_{K'} \frac{t(p,q,\lambda)}{c(p,\lambda,\nu)^\alpha}\, d\nu
		\\ +4\pi
		\Res_{\mu=\lambda+1-(\alpha\frac{1}{2})+2k}\left(
		(\tilde{A}_{(p,\lambda),(q,\mu)}^\alpha f,\tilde{A}_{(p,\lambda),(q,-\mu)}^\alpha \overline{f})_{K'} \frac{t(p,q,\lambda)}{c(p,\lambda,\mu)^\alpha}
		\right)
	\end{multline*}
	and that the residues are of the claimed form.

We prove the statement by induction.
Conisder the statement holding for $\Re\lambda\in (-\frac{1}{2}-2k,-2k)$ and consider the integral for $\alpha=+$ of \eqref{C:eq:plancherel_induction_start}
\begin{equation}\label{C:eq:integral_to_continue_+}
\int_{i\R} (\tilde{A}_{(p,\lambda),(q,\nu)}^+f,\tilde{A}_{(p,\lambda),(q,-\nu)}^+ \overline{f})_{K'} \frac{t(p,q,\lambda)}{c(p,\lambda,\nu)^+}\, d\nu.
\end{equation}
Then for $\Re\nu\in[0,\frac{1}{2}]$, and $\Re\lambda \in (-\frac{1}{2}-2k,-2k)$,
$c(p,\lambda,\nu)^+$  vanishes if and only if $\nu=\lambda+\frac{1}{2}+2k$ such that the integral has a simple pole. Then moving the contour of integration we obtain
\begin{multline*}
\int_{i\R+\frac{1}{2}} (\tilde{A}_{(p,\lambda),(q,\nu)}^+f,\tilde{A}_{(p,\lambda),(q,-\nu)}^+ \overline{f})_{K'} \frac{t(p,q,\lambda)}{c(p,\lambda,\nu)^+}\, d\nu
\\
+2\pi\Res_{\mu=\lambda+\frac{1}{2}+2k}\left(
(\tilde{A}_{(p,\lambda),(q,\mu)}^+f,\tilde{A}_{(p,\lambda),(q,-\mu)}^+ \overline{f})_{K'} \frac{t(p,q,\lambda)}{c(p,\lambda,\mu)^+}
\right).
\end{multline*}
Now for $\Re \nu=\frac{1}{2}$, $c(p,\lambda,\nu)^+$ does not vanish for $\lambda \in (-1-2k,-2k)$. On the other hand for $\Re \nu \in [0, \frac{1}{2}]$ and $\Re \lambda \in (-1-2k,-\frac{1}{2}-2k)$, $c(p,\lambda,\nu)^+$ vanishes only at $\nu=-\lambda-\frac{1}{2}-2k$.
Then moving the contour of integration back towards $\Re \nu=0$ 
we have for $\lambda \in (-1-2k,-\frac{1}{2}-2k)$
\begin{multline*}
\int_{i\R+\frac{1}{2}} (\tilde{A}_{(p,\lambda),(q,\nu)}^+f,\tilde{A}_{(p,\lambda),(q,-\nu)}^+ \overline{f})_{K'} \frac{t(p,q,\lambda)}{c(p,\lambda,\nu)^+}\, d\nu
\\=
\int_{i\R} (\tilde{A}_{(p,\lambda),(q,\nu)}^+f,\tilde{A}_{(p,\lambda),(q,-\nu)}^+ \overline{f})_{K'} \frac{t(p,q,\lambda)}{c(p,\lambda,\nu)^+}\, d\nu\\
-2\pi\Res_{\mu=-\lambda-\frac{1}{2}-2k}\left(
(\tilde{A}_{(p,\lambda),(q,\mu)}^+f,\tilde{A}_{(p,\lambda),(q,-\mu)}^+ \overline{f})_{K'} \frac{t(p,q,\lambda)}{c(p,\lambda,\mu)^+}
\right).
\end{multline*}

Consider the residue $$\Res_{\mu=\lambda+\frac{1}{2}+2k}\left(
(\tilde{A}_{(p,\lambda),(q,\mu)}^+f,\tilde{A}_{(p,\lambda),(q,-\mu)}^+ \overline{f})_{K'} \frac{t(p,q,\lambda)}{c(p,\lambda,\mu)^+}
\right).$$
Then by inserting the Knapp--Stein intertwiner we have
\begin{multline}\label{C:eq:res_equality}
\pi \Res_{\mu=\lambda+\frac{1}{2}+2k}\left(
(\tilde{A}_{(p,\lambda),(q,\mu)}^+f,\tilde{A}_{(p,\lambda),(q,-\mu)}^+ \overline{f})_{K'} \frac{t(p,q,\lambda)}{c(p,\lambda,\mu)^+}
\right)\\=\pi\Res_{\mu=\lambda+\frac{1}{2}+2k}\left(
(\tilde{A}_{(p,\lambda),(q,\mu)}^+f,T'_{q,\mu}\circ \tilde{A}_{(p,\lambda),(q,\mu)}^+ \overline{f})_{K'} \frac{t(p,q,\lambda)}{c(p,\lambda,\mu)^+t'(p,q,\mu)}
\right)
	\\=t(p,q,\lambda)c(p,\lambda,p,k)^+_{Res}\\ \times (C_{(p,\lambda),(q,\lambda+\frac{1}{2}+2k)}^+ f,T'^+_{p,\lambda+\frac{1}{2}+2k} \circ C_{(p,\lambda),(q,\lambda+\frac{1}{2}+2k)}^+ \overline{f})_{K'},
\end{multline}
which is holomorphic in $\lambda$ for $\Re\lambda <-\frac{1}{2}-2k$ by Lemma~\ref{C:lemma:scalars_positive}.
Similarly we have
\begin{multline*}
	-\pi \Res_{\mu=-\lambda-\frac{1}{2}-2k}\left(
	(\tilde{A}_{(p,\lambda),(q,\mu)}^+f,\tilde{A}_{(p,\lambda),(q,-\mu)}^+ \overline{f})_{K'} \frac{t(p,q,\lambda)}{c(p,\lambda,\mu)^+}
	\right)\\=-\pi t(p,q,\lambda)\Res_{\mu=-\lambda-\frac{1}{2}-2k}\left( \frac{1}{c(p,\lambda,-\mu)^+t'(p,p,-\mu)c_C(p,q,-\mu)^2}  \right)   \\ \times (T'_{p,\lambda+\frac{1}{2}+2k} \circ C_{(p,\lambda),(q,\lambda+\frac{1}{2}+2k)}^+ f, C_{(p,\lambda),(q,\lambda+\frac{1}{2}+2k)}^+ \overline{f})_{K'},
\end{multline*}
which coincides with \eqref{C:eq:res_equality}.
Hence
we obtain
\begin{multline*}
\int_{i\R} (\tilde{A}_{(p,\lambda),(q,\nu)}^+f,\tilde{A}_{(p,\lambda),(q,-\nu)}^+ \overline{f})_{K'} \frac{t(p,q,\lambda)}{c(p,\lambda,\nu)^+}\, d\nu
\\ +4\pi
\Res_{\mu=\lambda+\frac{1}{2}+2k}\left(
(\tilde{A}_{(p,\lambda),(q,\mu)}^+f,\tilde{A}_{(p,\lambda),(q,-\mu)}^+ \overline{f})_{K'} \frac{t(p,q,\lambda)}{c(p,\lambda,\mu)^+}
\right)
\end{multline*}
as the analytic continuation of \eqref{C:eq:integral_to_continue_+} on $\lambda \in (-1-2k, -\frac{1}{2}-2k)$.
Since $c(p,\lambda,\nu)^+$ does not vanish for $\Re \lambda \in (-\frac{1}{2}-2(k+1),-\frac{1}{2}-2k)$ and $\Re \nu=0$, this even defines an analytic continuation on $ \Re \lambda\in  (-\frac{1}{2}-2(k+1),-\frac{1}{2}-2k)$.
Now taking the limit along the real line towards $\lambda=-\frac{1}{2}-2k$ from the right proves the statement.
For  $$\int_{i\R} (\tilde{A}_{(p,\lambda),(q,\nu)}^-f,\tilde{A}_{(p,\lambda),(q,-\nu)}^- \overline{f})_{K'} \frac{t(p,q,\lambda)}{c(p,\lambda,\nu)^-}\, d\nu$$
the argument works in the same way.
\end{proof} 

To state the main result we formulate the residues $c(p,\lambda,q,k)^\pm_{Res}$ more explicitly. The following can be deduced by simple calculations using the definitions of the functions $c(p,\lambda,\nu)^\pm$, $t'(p,q,\nu)$ and $c_C(p,q,\nu)^\pm$.

By Lemma~\ref{C:lemma:scalars_positive} the formula of Proposition~\ref{C:prop:plancherel_general} immediately yields Plancherel formulas for the unitarizable representations occurring  in $\pi_{p,\lambda}^\pm$.
\begin{corollary}\label{C:cor:plancherel}
\begin{enumerate}[label=(\roman{*})]
\item Let $p=0$. For $\lambda\in (-\rho,0)\cup -\rho-\Z_{\geq 0}$ we have for $f\in \pi_{0,\lambda}^\pm$
\begin{multline*}
\norm{f}^2_{0,\lambda}=\frac{1}{4}\sum_{\alpha=+,-}\Bigg(\int_{i\R} \norm{\tilde{A}_{(0,\lambda),(0,\nu)}^\alpha f}_{L^2(K')}^2 \frac{\abs{t(0,0,\lambda)}}{c(p,\lambda,\nu)^\alpha}d\nu
\\
+\sum_{  k \in [0, \frac{-\lambda-1+(\alpha \frac{1}{2})}{2})\cap \Z       }  \abs{t(0,0,\lambda)}c(0,\lambda,0,k)_{Res}^\alpha\norm{C_{(0,\lambda),(0,\lambda+1-(\alpha\frac{1}{2})+2k)}^\alpha f}^2_{0,\lambda+1-(\alpha\frac{1}{2})+2k }    \Bigg).
\end{multline*}

\item For $0<p<n$. For $\lambda\in [-\abs{p-\rho},0)$ we have for $f\in \pi_{p,\lambda}^\pm$
\begin{multline*}
\norm{f}^2_{p,\lambda}=\frac{1}{4}\sum_{\alpha=+,-}\sum_{q=p-1,p} \Bigg(\int_{i\R} \norm{\tilde{A}_{(p,\lambda),(q,\nu)}^\alpha f}_{L^2(K')}^2 \frac{\abs{t(p,q,\lambda)}}{c(p,\lambda,\nu)^\alpha}d\nu
\\
+\sum_{  k \in [0, \frac{-\lambda-1+(\alpha \frac{1}{2})}{2})\cap \Z       }  \abs{t(p,q,\lambda)}c(p,\lambda,q,k)_{Res}^\alpha\norm{C_{(p,\lambda),(q,\lambda+1-(\alpha\frac{1}{2})+2k)}^\alpha f}^2_{q,\lambda+1-(\alpha\frac{1}{2})+2k }    \Bigg).
\end{multline*}
\item Let $p=n$. For $\lambda\in (-\rho,0)\cup -\rho-\Z_{\geq 0}$ we have for $f\in \pi_{n,\lambda}^\pm$
\begin{multline*}
\norm{f}^2_{n,\lambda}=\frac{1}{4}\sum_{\alpha=+,-}\Bigg(\int_{i\R} \norm{\tilde{A}_{(n,\lambda),(n-1,\nu)}^\alpha f}_{L^2(K')}^2 \frac{\abs{t(n,n-1,\lambda)}}{c(p,\lambda,\nu)^\alpha}d\nu
\\
+\sum_{  k \in [0, \frac{-\lambda-1+(\alpha \frac{1}{2})}{2})\cap \Z       }  \abs{t(n,n-1,\lambda)}c(n,\lambda,n-1,k)_{Res}^\alpha\\ \times \norm{C_{(n,\lambda),(n-1,\lambda+1-(\alpha\frac{1}{2})+2k)}^\alpha f}^2_{n-1,\lambda+1-(\alpha\frac{1}{2})+2k }    \Bigg).
\end{multline*}
\end{enumerate}
\end{corollary}
We remark that for $p\neq \frac{n}{2}$, $t(p,p-1,p-\rho)=t(p,p,\rho-p)=0$ while $t(p,q,\lambda)\neq 0$ in all other cases.
By the corollary above we immeditelly obtain unitary branching laws for the complementary series.
\begin{theorem}\label{C:theorem:branching_complementary series}
\begin{enumerate}[label=(\roman{*})]
\item For $p=0$ and $\lambda\in (-\rho,0)$ we have
$$\hat{\pi}_{0,\lambda}^\pm|_{G'}\simeq \bigoplus_{\alpha=+,-} \left(\int^\oplus_{i\R_+} \hat{\tau}_{0,\nu}^\alpha\, d\nu \oplus \bigoplus_{k \in [0, \frac{-\lambda-1+(\alpha \frac{1}{2})}{2})\cap \Z} \hat{\tau}_{0,\lambda+1-(\alpha \frac{1}{2})+2k}^{\pm\alpha}\right).$$
\item for $0<p<n$ and $\lambda\in (-\abs{\rho-p},0)$ we have
$$\hat{\pi}_{p,\lambda}^\pm|_{G'}\simeq\bigoplus_{\alpha=+,-}\bigoplus_{q=p-1,p} \left(\int^\oplus_{i\R_+} \hat{\tau}_{q,\nu}^\alpha\, d\nu \oplus \bigoplus_{k \in [0, \frac{-\lambda-1+(\alpha \frac{1}{2})}{2})\cap \Z} \hat{\tau}_{q,\lambda+1-(\alpha \frac{1}{2})+2k}^{\pm\alpha}\right).$$
\item For $p=n$ and $\lambda\in (-\rho,0)$ we have
$$\hat{\pi}_{n,\lambda}^\pm|_{G'}\simeq\bigoplus_{\alpha=+,-} \left(\int^\oplus_{i\R_+} \hat{\tau}_{n-1,\nu}^\alpha\, d\nu \oplus \bigoplus_{k \in [0, \frac{-\lambda-1+(\alpha \frac{1}{2})}{2})\cap \Z} \hat{\tau}_{n-1,\lambda+1-(\alpha \frac{1}{2})+2k}^{\pm\alpha}\right).$$
\end{enumerate}
\end{theorem}

Similarly we can deduce unitary branching laws for the unitary quotients with non-trivial $(\mathfrak{g},K)$-cohomology.
\begin{theorem}\label{C:theorem:branching_infinitesimal_char_rho}
\begin{enumerate}[label=(\roman{*})]
\item For the one dimensional unitary quotients we have
$$\widehat{\Pi}_{0,\pm}|_{G'} \simeq \widehat{\Pi}'_{0,\pm}, \qquad \widehat{\Pi}_{n+1,\pm}|_{G'}\simeq\widehat{\Pi}'_{n,\pm}.$$
\item For $0<p\leq \frac{n}{2}$ we have
$$\widehat{\Pi}_{p,\pm}|_{G'} \simeq \widehat{\Pi}'_{p,\pm}\oplus \bigoplus_{k\in (0,\rho'-p+1)\cap \Z} \hat{\tau}_{p-1,p-1-\rho'+k}^{\mp(-1)^k}\oplus \bigoplus_{\alpha=+,-} \int_{i\R_+}^\oplus \hat{\tau}_{p-1,\nu}^\alpha\, d\nu .
$$
\item For $p= \frac{n+1}{2}$ we have
$$\widehat{\Pi}_{\frac{n+1}{2},\pm}|_{G'} \simeq \bigoplus_{\alpha=+,-} \int_{i\R_+}^\oplus \hat{\tau}_{\frac{n-1}{2},\nu}^\alpha\, d\nu .
$$
\item For $\frac{n+1}{2}< p\leq n $ we have
$$\widehat{\Pi}_{p,\pm}|_{G'} \simeq \widehat{\Pi}'_{p-1,\pm}\oplus \bigoplus_{k\in (0,p-1-\rho')\cap \Z} \hat{\tau}_{p-1,\rho'-p+1+k}^{\pm(-1)^k}\oplus \bigoplus_{\alpha=+,-} \int_{i\R_+}^\oplus \hat{\tau}_{p-1,\nu}^\alpha\, d\nu .
$$
\end{enumerate}
\end{theorem}
\begin{proof}
The statement for the one-dimensional representations is clear.
For $0\leq p\leq \frac{n-1}{2}$  the statement follows from Corollary~\ref{C:cor:plancherel} in the following way.
Since $\norm{f}_{p,p-\rho}=\norm{\pr_{p-\rho} f}_{p,p-\rho,quo}$ for all $f\in \hat{\pi}_{p,p-\rho}^\pm$ the Plancherel formula of Corollay~\ref{C:cor:plancherel} is essentially the Plancherel formula for the quotient $\widehat{\Pi}_{p+1,\mp}$. Then the statement follows since $t(p,p-1, p-\rho)=0$ and $t(p,p,p-\rho)\neq 0$.
Similarly for $\frac{n+1}{2}\leq p \leq n$, and the quotient $\widehat{\Pi}_{p,\pm}$ of $\hat{\pi}_{p,\rho-p}^\pm$ since $t(p,p,\rho-p)=0$ and $t(p,p-1,\rho-p)\neq0$.
For $p=\frac{n}{2}$ we have that $$\Pi_{\frac{n}{2},\pm}:=\ker\left(T_{\frac{n}{2},0}- \frac{\pi^\frac{n}{2}}{(\frac{n}{2})!} \operatorname{id}\right)\subseteq \pi_{\frac{n}{2},0}^\pm$$
and $$\Pi_{\frac{n}{2}+1,\mp}:=\ker\left(T_{\frac{n}{2},0}+ \frac{\pi^\frac{n}{2}}{(\frac{n}{2})!} \operatorname{id}\right)\subseteq \pi_{\frac{n}{2},0}^\pm.$$
Then the statement follows from the functional equations Theorem~\ref{C:theorem:functional_equations}.
\end{proof}
In the same way as above we obtain branching laws for the unitarizable representations $I_{p,j,\pm}$ for $p=0,n$.
\begin{theorem}\label{C:theorem:branchin_p=0,n}
	\begin{enumerate}[label=(\roman{*})]
	\item For $p=0$ we have
	$$\hat{I}_{0,j,\pm}|_{G'} \simeq \widehat{\Pi}'_{1,\mp} \oplus \bigoplus_{k=1}^j\hat{I}'_{0,k, \pm (-1)^{k+j} }   \oplus \bigoplus_{k \in (0,\rho') \cap \Z} \hat{\tau}_{0,-\rho'+k}^{\pm(-1)^{k+j}} \bigoplus_{\alpha=+,-} \oplus \int^\oplus_{i\R_+} \hat{\tau}_{0,\nu}^\alpha\, d\nu .$$
	\item For $p=n$ we have
	$$\hat{I}_{n,j,\pm}|_{G'} \simeq \widehat{\Pi}'_{n-1,\pm} \oplus\bigoplus_{k=1}^j \hat{I}'_{n-1,k, \pm (-1)^{k+j} }   \oplus \bigoplus_{k \in (0,\rho') \cap \Z} \hat{\tau}_{n-1,-\rho'+k}^{\pm(-1)^{k+j}} \oplus \bigoplus_{\alpha=+,-} \int^\oplus_{i\R_+} \hat{\tau}_{n-1,\nu}^\alpha\, d\nu .$$
\end{enumerate}
\end{theorem}

\vspace{1cm}
\textsc{CW: Mathematical Sciences, Chalmers University of Technology, SE-412 96 Göteborg, Sweden}\par
\textit{E-Mail address:} \texttt{weiske@chalmers.se}
\end{document}